\documentclass[11pt]{amsart}

\usepackage[margin=1.6in]{geometry}

\usepackage[T1]{fontenc}
\usepackage{mathrsfs}  
\usepackage{amssymb}
\usepackage{dsfont}
\usepackage[all]{xy}
\usepackage{pb-diagram, pb-xy}
\usepackage{graphicx}
\usepackage[hidelinks]{hyperref}
\hypersetup{colorlinks=true} 

\newtheoremstyle{customthm}
{7mm} 
{\topsep}     
{\itshape}    
{8pt}         
{\bf}         
{.~---}       
{5pt plus 1pt minus 1pt} 
{\thmname{#1}\thmnumber{ #2}\thmnote{ (\normalfont #3)}}

\newtheoremstyle{customdef}
{7mm} 
{\topsep}     
{\normalfont} 
{8pt}         
{\bf}         
{.~---}       
{5pt plus 1pt minus 1pt} 
{\thmname{#1}\thmnumber{ #2}\thmnote{ (\normalfont #3)}}

\newtheoremstyle{custompar}
{7mm} 
{\topsep}     
{\normalfont} 
{8pt}         
{\bf}         
{}            
{5pt plus 1pt minus 1pt} 
{\thmname{#1}\thmnumber{ #2}\thmnote{ (\normalfont #3)}}

\theoremstyle{customthm}
\newtheorem{thm}{Theorem}[section]
\newtheorem{cor}[thm]{Corollary}

\newtheorem{lem}[thm]{Lemma}
\newtheorem{prop}[thm]{Proposition}
\numberwithin{equation}{thm}

\theoremstyle{customdef}
\newtheorem{defn}[thm]{Definition}

\newtheorem{ex}[thm]{Example}
\newtheorem{para}[thm]{}

\numberwithin{equation}{thm}

\newcommand{\abs}[1]{\vert#1\vert}

\newcommand{\eps}{\varepsilon}
\newcommand{\tq}{\: | \:}

\newcommand{\qqet}{\qquad\mbox{and}\qquad}
\newcommand{\II}{\mathds{1}}

\renewcommand{\div}{\mathrm{div}}

\DeclareMathOperator{\End}{End}
\DeclareMathOperator{\Ext}{Ext}
\DeclareMathOperator{\ev}{ev}

\DeclareMathOperator{\Gal}{Gal}
\DeclareMathOperator{\GL}{GL}

\DeclareMathOperator{\Hom}{Hom}

\DeclareMathOperator{\inv}{inv} 

\DeclareMathOperator{\PSL}{PSL} 
\DeclareMathOperator{\pr}{pr} %

\renewcommand{\Re}{\mathrm{Re}}
\newcommand{\dR}{\mathrm{dR}}

\DeclareMathOperator{\sing}{sing}%
\DeclareMathOperator{\SL}{SL}%
\renewcommand{\sl}{\mathfrak{sl}}%
\DeclareMathOperator{\Sp}{Sp}%
\DeclareMathOperator{\spec}{Spec}
\DeclareMathOperator{\SO}{SO}%

\def\resp{\text{resp.}\kern.3em}

\newcommand{\IA}{\mathbb{A}}
\newcommand{\IC}{\mathbb{C}}

\newcommand{\IG}{\mathbb{G}}

\newcommand{\IP}{\mathbb{P}}
\newcommand{\IQ}{\mathbb{Q}}
\newcommand{\IR}{\mathbb{R}}

\newcommand{\IZ}{\mathbb{Z}}

\newcommand{\cA}{\mathcal A}
\newcommand{\cB}{\mathcal B}
\newcommand{\cE}{\mathcal E}
\newcommand{\cG}{\mathcal G}
\newcommand{\cH}{\mathcal H}
\newcommand{\cM}{\mathcal M}
\newcommand{\cO}{\mathcal O}
\newcommand{\cR}{\mathcal R}

\newcommand{\sD}{\mathscr D}
\newcommand{\sL}{\mathscr L}

\newcommand{\fg}{\mathfrak g}

\newcommand{\FT}{\mathrm{FT}}
\newcommand{\Hyp}{\mathrm{Hyp}}

\def\eg{\textit{e.g.}\kern.3em}
\def\ie{\textit{i.e.}\kern.3em}
\def\loccit{\textit{loc.cit.}\kern.3em}

\newcommand{\Conn}{{\normalfont\textbf{Conn}}}
\newcommand{\bE}{{\normalfont\textbf{E}}}
\newcommand{\bG}{{\normalfont\textbf{G}}}
\newcommand{\bH}{{\normalfont\textbf{H}}}
\newcommand{\RS}{\normalfont\textbf{RS}}%
\newcommand{\bT}{{\normalfont\textbf{T}}}

\def\sfrac#1#2{{#1}/{#2}}

\renewcommand{\mapsto}{\longmapsto}

\newcommand{\floor}[1]{\left\lfloor#1\right\rfloor}
\renewcommand{\Re}{\mathrm{Re}}

\DeclareMathOperator{\Sum}{sum}

\newcommand{\rH}{\mathrm{H}}

\renewcommand{\leq}{\leqslant}
\renewcommand{\geq}{\geqslant}

\title{A non-hypergeometric $E$-function}
\thanks{The research of J.\,F. was partially supported by the grant ANR-18-CE40-0017 of Agence Nationale de la Recherche.}

\author{Javier Fres\'an} 
\address{Javier Fres\'an, CMLS, \'Ecole polytechnique, Palaiseau, France}
\email{javier.fresan@polytechnique.edu}

\author{Peter Jossen}%
\address{Peter Jossen, ETH Z\"urich, Z\"urich, Switzerland and King's College, London, England}
\email{peter.jossen@math.ethz.ch}

\begin{document}

\begin{abstract}
We answer in the negative Siegel's question whether all $E$\nobreakdash-functions are polynomial expressions in hypergeometric $E$\nobreakdash-functions. Namely, we show that if an irreducible differential operator of order three annihilates an $E$-function in the hypergeometric class, then the singularities of its Fourier transform are constrained to satisfy a symmetry property that generically does not hold. The proof relies on André's theory of $E$-operators and Katz's computation of the Galois group of hypergeometric differential equations. 
\end{abstract}

\maketitle

\tableofcontents

\section*{Introduction and overview}

\begin{par}
With the goal of generalising the theorems of Hermite, Lindemann, and Weierstrass about transcendence of values of the exponential function, Siegel introduced the notion of \emph{$E$\nobreakdash-function} in his landmark 1929 paper \cite{Siegel1929}. An $E$\nobreakdash-function is a power series with algebraic coefficients that satisfies a linear differential equation and certain growth conditions. A paradigmatic example besides the exponential is the classical Bessel function $J_0(z)$, for which Siegel achieved his goal by proving that the values of $J_0(z)$ and $J_0'(z)$ at all non-zero algebraic numbers are algebraically independent. Siegel's methods were in principle suited for all $E$-functions satisfying a certain ``normality'' property, but its actual verification remained elusive in new examples. In removing this assumption in 1959, Shidlovskii \cite{SiegelShidlovsky} took a decisive step in the understanding of~$E$\nobreakdash-functions. Further amendments to what quickly became known as the Siegel--Shidlovskii theorem were later given by André \cite{GevreyII} and Beukers~\cite{beukers}.
\end{par}

\begin{par}
A rich class of $E$\nobreakdash-functions, including the motivating examples of the exponential and the Bessel function, is given by the series
\[
F\biggl(\begin{matrix} a_1, \ \ldots, \ a_p \\ b_1, \ \ldots, \ b_q \end{matrix} \:\bigg|\:\: \lambda z^{q-p} \biggr)=\sum_{n=0}^\infty \frac{(a_1)_n\cdots (a_p)_n}{(b_1)_n\cdots (b_q)_n}\lambda^n  z^{n(q-p)}
\]
for integers $0 \leq p<q$, rational parameters $a_1, \ldots, a_p \in \IQ$ and $b_1, \ldots, b_q \in \IQ \setminus \IZ_{\leq 0}$, and an algebraic scalar~$\lambda$. Here and throughout, $(x)_n=x(x+1)\cdots (x+n-1)$ denotes the rising Pochhammer symbol of a rational number $x$. We call these series \emph{hypergeometric $E$\nobreakdash-functions} of type $(p, q)$. After having proved that any polynomial expression with algebraic coefficients in $E$\nobreakdash-functions is again an~$E$\nobreakdash-function, Siegel observed that it would not be without interest to find an example of an~$E$\nobreakdash-function that does not come from hypergeometric $E$\nobreakdash-functions in this way \cite[p.\,225]{Siegel1929}. The question was still open twenty years later when he published his lecture notes on transcendental numbers~\hbox{\cite[p.\,58]{Siegel}}, and appears in Shidlovskii's book \cite[p.\,184]{Shidlovsky} as well. We reformulate it as follows: 
\end{par}

\vspace{2mm}
\begin{par}\emph{Does the $\overline \IQ[z]$-algebra generated by hypergeometric $E$\nobreakdash-functions contain all $E$\nobreakdash-functions?}
\end{par}

\vspace{2mm}
\begin{par}
According to a theorem of Gorelov \cite{gorelov1, gorelov2, RivRoq}, $E$\nobreakdash-functions satisfying a differential equation of order at most two are indeed polynomial expressions in hypergeometric $E$\nobreakdash-functions. However, compelling evidence against a general positive answer was recently given by Fischler and Rivoal \cite{RivoalFischler}. They show that any special value of the closely related class of $G$\nobreakdash-functions arises as a coefficient in the asymptotic expansion at infinity of an $E$\nobreakdash-function. Since those of hypergeometric $E$\nobreakdash-functions are of a very special nature, a positive answer would yield an inclusion of the set of such values into a rather small set, which is then seen to contradict a form of Grothendieck's period conjecture. A priori, there is no lack of examples of $E$\nobreakdash-functions beyond those coming from hypergeometric functions. Given an algebraic variety $X$ over $\overline \IQ$ together with a regular function $f\colon X\to \IA^1$, one looks at exponential period functions
\[ P(z)=\int_\gamma e^{-zf}\omega \]
in the complex variable $z$, where $\omega$ is an algebraic differential form on $X$ and $\gamma$ is a topological cycle on a compactification of $X(\IC)$ that goes to infinity only in the directions where the real part of $f$ is positive, to ensure that the integral converges on a half-plane. Such functions are expected to be complex linear combinations of functions of the form $z^a\log(z)^bE(z)$, with $a\in \IQ$, $b\in \IZ_{\geq 0}$, and~$E$ an $E$\nobreakdash-function, and one might even speculate, in the spirit of the Bombieri--Dwork conjecture, that \emph{all} $E$\nobreakdash-functions arise from geometry in this manner. The one given in the following theorem was produced by taking the polynomial $f(x)=x^4-x^2+x$, viewed as a regular function on~$X=\IA^1$, the differential form $\omega = dx$, and the real line as $\gamma$.
\end{par}

\vspace{4mm}
\begin{par}{\bf Theorem.} \emph{The answer to Siegel's question is negative. For example,}
\[ \sum_{n=0}^\infty \biggl(\sum_{m=0}^{\floor{2n/3}} \frac{(\tfrac 14)_{n-m} }{(2n-3m)!(2m)!}\biggr) \:z^n \] 
\emph{is an $E$-function that is transcendental over the $\overline \IQ[z]$-algebra generated by hypergeometric $E$\nobreakdash-functions.}
\end{par}

\vspace{4mm}
\begin{par}
Let us outline the main ideas of the proof. Every $E$\nobreakdash-function is annihilated by a special type of differential operator called \emph{$E$-operator}. For hypergeometric $E$\nobreakdash-functions, such an operator can be obtained by appropriately modifying a classical hypergeometric differential equation. Let us write~$\cE$ for the differential $\IC$-algebra generated by all solutions of $E$-operators, and $\cH$ for the subalgebra generated by all solutions of hypergeometric $E$-operators. Our goal is to show that the inclusion $\cH \subseteq \cE$ is strict and, more precisely, that the above function belongs to~$\cE$ but not to $\cH$. As a first step, we reformulate this inclusion of algebras in terms of an inclusion of certain categories of~$\sD$\nobreakdash-modules on~$\IG_m$, namely the categories $\bH$ and $\bE$ of those having a basis of solutions in the differential algebras $\cH$ and $\cE$ respectively. By Andr\'e's fundamental results on $E$-operators~\cite{GevreyI}, these are tannakian subcategories of the category of vector bundles with connection on~$\IG_m$, and every $E$\nobreakdash-function is a solution of some module in $\bE$. Our updated goal is then to show that the inclusion of categories $\bH\subseteq \bE$ is not an equivalence and that the given power series is a solution of some object of $\bE$ but not of any object of~$\bH$. 
\end{par}

\begin{par}
Given a $\sD_{\IG_m}$-module $M$ in $\bE$, we can pushforward it via the inclusion $j\colon\IG_m\to \IA^1$ and then take its Fourier transform. It turns out that the so obtained $\sD_{\IA^1}$-module $A=\FT(j_\ast M)$ is regular singular at each of its finitely many singularities, including infinity, and has vanishing de Rham cohomology. Such $\sD$-modules form a tannakian category, the role of tensor product being played by additive convolution, and the tannakian Galois group of $A$ is isomorphic to the differential Galois group of $M$. The key observation at this stage is that, if~$M$ belongs to~$\bH$, the singularities of $A$ tend to form a rather particular configuration. For instance, if~$M$ is the module associated with the differential equation for a hypergeometric $E$\nobreakdash-function of type~$(p,q)$, then~$A$ has dimension $q$ and its non-zero singularities lie on a regular \hbox{$(q-p)$-gon}. If, moreover,~$M$ is simple, its Galois group is well understood by a theorem of Katz \cite{Katz}. Combining this result with some tannakian shenanigans, we reach the conclusion that if~$M$ is a three-dimensional object of $\bH$ whose Galois group contains $\SL_3$, then the singularities of $A$ are either collinear or form an equilateral triangle.
\end{par}

\begin{par}
With these preparations in place, we construct a rich family of objects of $\bE$ by considering the $\sD_{\IG_m}$-module $M$ that is naturally associated with the parameter integral
\[
P(z) = \int_{\IR} e^{-zf(x)}dx,
\]
where $f$ is a monic polynomial of degree four with algebraic coefficients. The function~$P$ satisfies a differential equation of order three, so $M$ has rank three. A computation reveals that~$A=\FT(j_\ast M)$ is the $\sD_{\IA^1}$-module whose solutions are the algebraic functions $u(z)$ satisfying the equation $f(u(z))=z$, from which it follows that $M$ belongs to $\bE$ and that the singularities of~$A$ are the critical values of $f$. For most choices of the polynomial $f$ (in particular, we can verify this for~$f(x)=x^4-x^2+x$), the module~$M$ is simple with Galois group~$\GL_3$ and the three singularities of $A$ are neither collinear nor do they form an equilateral triangle. We conclude that~$M$ does not belong to $\bH$ and that every non-zero solution of $M$ (in particular, the function~$P$) is transcendental over~$\cH$. It is then not hard to check that the same holds for the~$E$\nobreakdash-functions~$E_0$ and $E_2$ uniquely determined~by the monodromy decomposition 
\[
P(z) = \int_{\IR} e^{-z(x^4-x^2+x)}dx=\tfrac{1}{2}\Gamma\bigl(\tfrac 14\bigr) z^{-1/4}E_0(z) + \tfrac{1}{2}\Gamma\bigl(-\tfrac 14\bigr)z^{1/4}E_2(z). 
\]
Indeed, since $z^{-1/4}$ and $z^{1/4}$ are linearly independent over the field of Laurent series, any differential operator with polynomial coefficients that annihilates $P$ must also annihilate $z^{-1/4}E_0(z)$ and $z^{1/4}E_2(z)$, and if those functions are transcendental over $\cH$, then so are $E_0$ and $E_2$. The function~$E_0$ is the one given in the theorem. Although our method provides a host of examples of non\nobreakdash-hypergeometric $E$-functions, we do not yet fully understand their structure. In particular, finding a set of necessary and sufficient conditions for a three-dimensional simple object of~$\bE$ to belong to~$\bH$ remains an open question. 
\end{par}

\vspace{4mm}
\begin{par}{\bf Acknowledgements.} Our interest in Siegel's question was aroused by Rivoal's survey~\cite{Rivoal} and his joint work with Fischler; it is our pleasure to thank them both for inspiration and useful discussions. Many thanks as well to Erik Panzer for explaining to us how to perform the computations of Section \ref{sec:ComputationExponentialIntegral}, and to Boris Adamczewski, Yves André, Daniel Bertrand, Thomas Preu, and the anonymous referees for their comments on previous versions.
\end{par}

\vspace{14mm}
\section{\texorpdfstring{On $E$-functions and their differential equations}{On E-functions and their differential equations}}

\begin{par}
In this section, we recall the notions of $E$\nobreakdash-function and $G$-series, as well as some theorems concerning the differential equations they satisfy, due to Andr\'e, Chudnovsky, and Katz. Throughout, $\overline \IQ$ denotes the algebraic closure of $\IQ$ in $\IC$.
\end{par}

\begin{para}
\begin{par}
As outlined in the introduction, we will deal with differential algebras generated by solutions of various classes of differential operators. All these algebras will be contained in the differential algebra of \emph{convergent Laurent series with monodromy}
\[
\cM=\IC\{\!\{z\}\!\}[(z^a)_{a\in \IC}, \log(z)],
\]
which is an integral domain and contains an algebraic closure of the field $\IC(z)$ of rational functions. Its elements can be written~as 
$$f(z) =\sum_{i=1}^n z^{a_i}\log(z)^{b_i} F_i(z),$$
where the $F_i$ are Laurent series with positive radius of convergence, $a_i\in \IC$, and~$b_i\in \IZ_{\geq 0}$. This representation is unique if one further imposes the constraint $0\leq \Re(a_i) < 1$. We can interpret~$f$ as a holomorphic function on a sector  $\{z\in \IC \tq \Re(z)>0 \mbox{ and } \abs{z}<r\}$ for some small $r>0$.
\end{par}

\begin{par}
By a \emph{differential operator} we generally understand an element of the Weyl algebra $\overline \IQ[z,\partial]$ or occasionally of $\overline \IQ(z)[\partial]$. Fuchs's criterion and Frobenius's method show that a differential operator $L \in \overline \IQ(z)[\partial]$ has a regular singularity at $z=0$ if and only if $L$ admits a basis of solutions in $\cM$. Given a non-constant rational function $h\in \overline \IQ(z)$, we write
$$[h]^\ast\colon\overline \IQ(z)[\partial] \to \overline \IQ(z)[\partial]$$
for the ring homomorphism determined by $[h]^\ast(z)=h(z)$ and $[h]^\ast(\partial)=\tfrac 1{h'(z)}\partial$. If $L\in\overline \IQ(z)[\partial]$ is regular singular at $0$ and $h(0)=0$, then $[h]^\ast L$ is again regular singular at $0$, and the solutions of $[h]^\ast L$ are all of the form $f\circ h$ for a solution $f$ of $L$. Note that a determination of the logarithm needs to be chosen in order to interpret $f\circ h$ as a Laurent series with monodromy.
\end{par}

\begin{par}
We call \emph{Fourier transform} for differential operators the $\overline \IQ$-linear ring automorphism
\begin{equation}\label{Eqn:FTasRingAutomorhpismOfD}
\FT\colon \overline \IQ[z,\partial]\to \overline \IQ[z,\partial]
\end{equation}
determined by $\FT(z)=-\partial$ and $\FT(\partial)=z$, and \emph{adjunction} the $\overline \IQ$-linear map
\begin{equation*}\label{Eqn:AdjunctionOfOperators}
(-)^\ast \colon \overline \IQ[z,\partial]\to \overline \IQ[z,\partial]
\end{equation*}
defined on monomials by $(z^n\partial^m)^\ast = (-\partial)^mz^n$. The latter is an involution, in the sense that the equalities $(L^\ast)^\ast = L$ and $(L_1L_2)^\ast = L_2^\ast L_1^\ast$ hold for all differential operators $L, L_1, L_2$. 
\end{par}
\end{para}

\begin{defn}\label{Def:EFunctionGFunction}
An \emph{$E$\nobreakdash-function} $E(z)$, respectively a \emph{$G$\nobreakdash-series} $G(z)$, is a formal power series with algebraic coefficients of the form
\begin{displaymath}
E(z)=\sum_{n=0}^\infty \frac{a_n}{n!}z^n, \quad \text{respectively} \quad G(z)=\sum_{n=0}^\infty a_n z^n, 
\end{displaymath} 
that is annihilated by some non-zero differential operator $L \in \overline \IQ[z, \partial]$, and whose coefficients~$a_n$ satisfy the following growth condition: letting $d_n \geq 1$ denote, for each $n \geq 0$, the smallest integer such that $d_na_0, \ldots, d_na_n$ are algebraic integers, there exists a real number~$C>0$ such that the inequalities $\abs{\sigma({a_n})} \leq C^n$ and $d_n \leq C^n$ hold\footnote{In Siegel's original definition of $E$-functions~\cite[p.\,223]{Siegel1929}, the coefficients and their denominators are required to grow at most as $(n!)^\eps$ for any $\eps>0$. In the presence of a differential equation, this condition is believed to be equivalent to the more restrictive one from Definition~\ref{Def:EFunctionGFunction}, but this is still unknown. At any rate, our example also answers in the negative Siegel's question for his a priori larger class of $E$-functions. } for all $n\geq 1$ and all $\sigma\in\Gal(\overline \IQ/\IQ)$.
\end{defn}

\vspace{4mm}
\begin{para}
\begin{par}
Setting all coefficients $a_n$ equal to $1$ yields the exponential series as an example of an~$E$\nobreakdash-function and the geometric series as an example of a $G$\nobreakdash-series, thus the names. According to a theorem of Eisenstein \cite{Eisenstein}, if a power series $\sum a_nz^n \in \overline \IQ[\![z]\!]$ is algebraic over $\overline \IQ(z)$, then there exists an integer $d \geq 1$ such that $d^na_n$ is an algebraic integer for all $n\geq 1$; since such series also have a positive  radius of convergence and satisfy a differential equation, they are examples of $G$\nobreakdash-series. The set of $E$\nobreakdash-functions contains the ring of polynomials $\overline \IQ[z]$ and is stable under sums, products, and derivation, so $E$\nobreakdash-functions form a differential subalgebra of~$\overline \IQ[\![ z]\!]$. The existence of a non-trivial differential equation annihilating the series $E(z)$ implies that the coefficients~$a_n$ generate a finite field extension of $\IQ$, a fact that can be used to show that any~$E$\nobreakdash-function is a $\overline \IQ$-linear combination of $E$\nobreakdash-functions with rational coefficients; the same is true for~$G$\nobreakdash-series as well. It follows from the growth condition on the coefficients and their denominators that~$E$\nobreakdash-functions have infinite radius of convergence, so they can be interpreted as entire functions, whereas $G$\nobreakdash-series have a positive but finite radius of convergence unless they are polynomials. In fact, the intersection of $E$\nobreakdash-functions and $G$\nobreakdash-series is the ring of polynomials with algebraic coefficients. If $G(z)$ is a $G$\nobreakdash-series and $h(z) \in \overline \IQ(z)$ is a rational function with~$h(0)=0$, then the composite $G(h(z))$ is again a $G$\nobreakdash-series. In contrast, if $E(z)$ is an~$E$\nobreakdash-function, then so is $E(\lambda z)$ for $\lambda \in \overline\IQ$, but not $E(z^n)$ for $n\geq 2$ unless $E$ is a polynomial.
\end{par}
\end{para}

\begin{para}\label{Par:IntroductionAlgebrasRHE}
\begin{par}
Following Andr\'e \cite{GevreyI}, we consider the differential subalgebras $\cE$ and $\cG$ of~$\cM$ consisting of those functions
\begin{equation}\label{Eqn:FormOfElementsOfE}
f(z) = \sum_{i=1}^n c_iz^{a_i}\log(z)^{b_i} F_i(z)    
\end{equation}
with \emph{rational} exponents $a_i\in \IQ$, integers $b_i\geq 0$, complex coefficients $c_i$, and where the~$F_i$ are~$E$\nobreakdash-functions and  $G$\nobreakdash-series respectively. In~\cite{GevreyI}, elements of $\cE$ are called \emph{holonomic arithmetic Nilsson\nobreakdash--Gevrey series} of order~$-1$, the word \emph{order} referring to the power $(n!)^{-1}$ in the definition of $E$\nobreakdash-functions; accordingly, elements of~$\cG$ are called \emph{holonomic arithmetic Nilsson\nobreakdash--Gevrey series} of order $0$. Algebraic functions over~$\overline \IQ(z)$ can be viewed as elements of $\cG$, in the sense that every convergent Laurent series with monodromy that is algebraic over $\overline \IQ(z)$ belongs to~$\cG$. Lemma~\ref{Lem:IndependenceSeriesAndE} below shows that any~$\IC$\nobreakdash-linear combination of $E$\nobreakdash-functions that has algebraic coefficients is itself an $E$\nobreakdash-function, and similarly for $G$\nobreakdash-series. Hence, the expression \eqref{Eqn:FormOfElementsOfE} is unique up to the obvious modifications.
\end{par}
\end{para}

\begin{lem}\label{Lem:IndependenceSeriesAndE}
Let $K$ be a subfield of $\IC$ and $V\subseteq K[\![z]\!]$ a $K$-linear subspace. The equality
$$(\IC\otimes_K V) \cap K[\![z]\!] = V$$
holds, the intersection taking place in $\IC[\![z]\!]$. 
\end{lem}

\begin{proof}
The inclusion $\supseteq$ is obvious. Conversely, each element $f\hspace{-.5mm}\in\hspace{-.5mm} \IC\otimes_K V \subseteq \IC[\![z]\!]$ can be written~as
$$f = \lambda_0f_0 + \sum_{n=1}^N \lambda_nf_n,$$
where $\lambda_0=1, \lambda_1, \ldots,\lambda_N$ are $K$-linearly independent complex numbers, and $f_0,\ldots,f_N$ are elements of $V$. Denote by $c_k$ and by $c_{nk}$ the coefficients of $z^k$ in $f$ and $f_n$ respectively. If we now assume that $f$ lies in $K[\![z]\!]$, we obtain $K$-linear relations
\[
0 = \lambda_0(c_{0k}-c_k) +  \sum_{n=1}^N \lambda_nc_{nk},
\]
from which the equalities $c_{0k}=c_k$ and $c_{nk}=0$ for $n\geq 1$ follow. Hence, $f=f_0$ lies in $V$.
\end{proof}

\begin{defn}\label{Def:GOperatorEoperator}
A differential operator $L \in \overline \IQ(z)[\partial]$ is said to be a \emph{$G$-operator} if $L$ admits a basis of solutions in the algebra $\cG$. A differential operator $L\in \overline \IQ[z,\partial]$ is said to be an \emph{$E$-operator} if its Fourier transform $\FT(L)$ is a $G$-operator.
\end{defn}

\begin{para}\label{Par:DefinitionsOfGoperator}
\begin{par}
We should explain why this definition of $G$-operators is equivalent to other definitions found in the literature. First of all, since $\cG$ contains the field of rational functions~$\overline \IQ(z)$, an operator $L$ is a $G$-operator if and only if $PLQ$ is a $G$-operator for all non-zero rational functions~$P, Q\in \overline\IQ(z)$. In particular, we can multiply any $G$-operator $L \in \overline\IQ(z)[\partial]$ with an appropriate polynomial on the left in order to obtain a $G$-operator with polynomial coefficients, whose Fourier transform is then an $E$-operator. Since $G$-series are stable under derivation, every left or right factor of a $G$-operator is again a $G$-operator. From Fuchs's criterion, it follows that a $G$-operator is at most regular singular at $z=0$ with rational exponents, and hence admits at least one non-zero solution of the form~$z^af(z)$ with $a\in \IQ$ and a power series~$f\in \overline\IQ[\![z]\!] \cap \cG$, which is then a $G$\nobreakdash-series by Lemma~\ref{Lem:IndependenceSeriesAndE}.
\end{par}

\begin{par}
A commonly used definition, for instance by André \cite[IV\,5.2, p.\,76]{Gfunctions}, is that $L\in \overline\IQ(z)[\partial]$ is a~$G$\nobreakdash-operator if it satisfies Galochkin's condition, or equivalently Bombieri's condition 
$$\mathrm{Galochkin:}\:\:\:\sigma(L)<\infty \qquad \Longleftrightarrow \qquad \mathrm{Bombieri:}\:\:\:\rho(L)<\infty.$$
The \emph{size} $\sigma(L) \in \IR_{\geq 0}\cup\{\infty\}$ measures the growth in height of the coefficients of formal solutions of $L$, whereas the \emph{radius} $\rho(L) \in \IR_{\geq 0}\cup\{\infty\}$ measures the radii of convergence of $p$-adic solutions. The equivalence of these two conditions is a difficult theorem in itself; see~\hbox{\cite[IV, Th.\,5.2]{Gfunctions}} or~\hbox{\cite[VII, Th.\,2.1]{Dwork}.} If $L$ has finite size, it admits a basis of solutions in the algebra $\cG$ by \hbox{\cite[V\,6.6, Cor.]{Gfunctions}}, and hence is a $G$-operator in our sense too.
\end{par}

\begin{par}
Conversely, a theorem of Chudnovsky \cite[Th.\,III]{chudnovsky} says that an operator of minimal order that annihilates a given non-zero $G$\nobreakdash-series has finite size; see \cite[VIII, Th.\,1.5]{Dwork}. In particular, all irreducible~$G$\nobreakdash-operators have finite size. Taking into account that the size of a product of operators is bounded by the sizes of the factors (\cite[IV\,4.1, Lem\,2 and IV\,4.2, Prop.]{Gfunctions}), it follows that all~$G$\nobreakdash-operators have finite size. Both definitions are therefore equivalent. 
\end{par}
\end{para}

\begin{thm}[André, Chudnovsky, Katz]\label{Thm:PropertiesOfGoperators}  $G$-operators satisfy the following properties: 
\begin{enumerate}
    \item Every $G$-series, and more generally, every element of $\cG$ is annihilated by a $G$\nobreakdash-ope\-ra\-tor.
    \item Products and adjoints of $G$-operators are $G$-operators, and every left or right factor of a $G$-operator is a $G$-operator.
    \item $G$-operators have regular singularities on $\IP^1$, all with rational local exponents.
    \item If $L$ is a $G$-operator, then so is $[h]^\ast L$ for any non-constant rational function $h\in \overline\IQ(z)$.
\end{enumerate}
\end{thm}

\begin{proof}

\begin{par} Once we know that the various definitions of $G$-operators are equivalent, statement (1) for $G$-series is due to Chudnovsky; it extends to all elements of $\cG$ since $z^a\log(z)^b$ is a solution of a $G$\nobreakdash-operator for $a \in \IQ$ and $b \in \IZ_{\geq 0}$, and sums and products of solutions of~$G$\nobreakdash-operators are again annihilated by $G$-operators \cite[Lem.\,3.6.1]{GevreyI}; see the case $s=0$ of Andr\'e's purity theorem \cite[p.\,706]{GevreyI}. More precisely, any non-zero operator of minimal order annihilating a given element of $\cG$ is a $G$-operator.
\end{par}

\begin{par}
Statement (2), but formulated in terms of radii of differential modules, is \cite[IV\,3.3, Lem.\,2]{Gfunctions}. To translate from modules to operators, we first notice that the size of an operator is defined in loc.\,cit. as the size of the associated module. With that in mind, it suffices to observe that with any product of operators $L=L_1L_2$ is associated a short exact sequence of differential modules, and that if the generic fibre of a $\sD_{\IA^1}$-module is given by an operator $L\in \overline\IQ[z, \partial]$, then the generic fibre of its dual is given by the adjoint operator $L^\ast$; see \cite[1.5]{KatzGalois}.
\end{par}

\begin{par}
Statement (3) relies on Katz's global nilpotence theorem \hbox{\cite[III, Th.\,6.1]{Dwork}}, which says that a differential operator with nilpotent $p$-curvature for almost all $p$ has regular singularities with rational exponents. To apply it, one first shows using Bombieri's condition that the generic radius of convergence of solutions of a $G$-operator at a place dividing $p$ is bigger than~$p^{-1/(p-1)}$ for almost all~$p$, and then that this bound implies the nilpotence of $p$-curvature; see \cite[IV\,5.3]{Gfunctions}. 
\end{par}

\begin{par}
As to statement (4), let $L$ be a $G$-operator and first suppose that the rational function~$h$ satisfies $h(0)=0$. The composite of any $G$\nobreakdash-series with $h$ is again a $G$\nobreakdash-series, and we can interpret $\log(h(z))$ and $h(z)^a$ as elements of $\cG$. Hence, if $g \in \cG$ is a non-zero solution of~$L$, then the composite $g(h(z)) \in \cG$ is a non-zero solution of $[h]^\ast L$, so $[h]^\ast L$ admits indeed a basis of solutions in $\cG$. To conclude, it suffices to show that statement (4) also holds for~$h(z)=z^{-1}$ and $h(z)=z-a$ with $a\in \overline \IQ$. Since in these cases $h$ is invertible and in view of statement~(2) we can suppose that $L$ is irreducible, this follows from Andr\'e's permanence theorem \cite[p.\,706]{GevreyI}.
\end{par}
\end{proof}

\begin{para}\label{Par:PermanenceProperty}
In the introduction to \cite{Dwork}, Dwork et al.\@ define $G$-operators as those $L\in \overline \IQ(z)[\partial]$ that admit a basis of solutions consisting of $G$\nobreakdash-series around some algebraic point. In view of statement~(4) of Theorem \ref{Thm:PropertiesOfGoperators}, this is coherent with the previous definitions. Besides, the statement for the map $h(z)=z^{-1}$ can be rephrased as follows. Pick $g \in \cG$, and regard $g\circ h$ as a holomorphic function on the half-plane $\{z\in \IC\tq \Re(z)>C\}$ for some sufficiently large~\hbox{$C>0$}. Let $L$ be a $G$-operator annihilating $g$, so that $[h]^\ast L$ annihilates $g\circ h$. The function $g\circ h$ can be analytically continued to any simply connected open subset of $\IC$ that avoids the finitely many singularities of $[h]^\ast L$. In particular, once we decide how to circumvent them, we can extend~$g \circ h$ to a sector $\{z\in \IC\tq \Re(z)>0 \mbox{ and } \abs{z}<r\}$ for some sufficiently small $r>0$. Statement~(4) of Theorem \ref{Thm:PropertiesOfGoperators} implies that this analytic continuation of $g\circ h$ belongs to $\cG$.
\end{para}

\begin{thm}[André]\label{Thm:PropertiesOfEoperators} $E$-operators satisfy the following properties: 
\begin{enumerate}
    \item Every $E$\nobreakdash-function, and more generally, every element of $\cE$ is annihilated by an $E$\nobreakdash-ope\-ra\-tor. Conversely, every $E$-operator admits a basis of solutions in $\cE$. 
    
    \item Products and adjoints of $E$-operators are again $E$-operators. 
    
    \item The non-trivial singularities of an $E$-operator are contained in $\{0,\infty\}$, and $0$ is at worst a regular singularity, whereas $\infty$ is in general irregular with slopes in $\{0, 1\}$. 
\end{enumerate}
\end{thm}

\begin{proof}
The first part of statement (1) is \cite[Th.\,4.2]{GevreyI} and the second one is the case \hbox{$s=-1$} of Andr\'e's purity theorem \cite[Th.\,4.3\,(iii)]{GevreyI}. Statement (2) follows from the analogous one of Theorem \ref{Thm:PropertiesOfGoperators}, since Fourier transform is a ring morphism, which is compatible with adjoints in that the equality 
\[
\FT(L)^\ast=[-z]^\ast\FT(L^\ast)
\] holds. Finally, statement (3) is \cite[Th.\,4.3\,(i),\,(ii),\,(iv)]{GevreyI}. Recall that a singularity of an operator $L$ is said to be trivial whenever $L$ admits a basis of holomorphic solutions around the point in question. 
\end{proof}

\begin{para}\label{Par:LaplaceTransform}
We now explain the relation between the Fourier transform of differential operators and the Laplace transform. Classically, the \emph{Laplace transform} of a suitable holomorphic function $f$ is the complex-valued function $\mathscr L f$ given by the integral
\begin{equation}\label{Eqn:LaplaceTransform}
(\mathscr L f)(z) = \int_0^\infty f(w)e^{-zw}dw,
\end{equation}
which is assumed to converge on the half-plane $\Re(z)> C$ for some large enough real number~$C$. The \emph{inverse Laplace transform} of a holomorphic function $g$ on $\Re(z)> C$ is given~by
$$(\mathscr L^{-1} g)(z) = \frac1{2\pi i}\int_{C'-i\infty}^{C'+i\infty}g(w)e^{zw}dw$$
for any $C'>C$, again under appropriate convergence conditions. Using integration by parts and differentiating under the integral sign, we obtain the relations
$$\sL(zf) = -\partial\sL(f) \qqet \sL(\partial f) = z\sL(f)-f(0),$$
provided the constant term $f(0) = \lim_{z\to 0}f(z)$ exists. The integral \eqref{Eqn:LaplaceTransform} converges for functions~$f \in \cE$ of the form $f(z)=z^a\log(z)^b F(z)$ for some $E$-function $F = \sum_{n=0}^\infty a_n z^n/n!$ and $a>-1$, with the same constant $C$ as in Definition \ref{Def:EFunctionGFunction}, and is explicitly given by
$$(\sL f)(z) = z^{-a-1}\sum_{k=0}^b \binom{b}{k}\log(z^{-1})^k\sum_{n=0}^\infty \frac{a_n \Gamma^{(b-k)}(a+n+1)}{n!} z^{-n},$$
where $\Gamma^{(m)}$ stands for the $m$-th derivative of the gamma function. Since the inner sum in this expression is a $\IC$-linear combination of $G$-series in the variable $z^{-1}$, taking the permanence property discussed in \ref{Par:PermanenceProperty} into account, we may regard $\mathscr L$ as a $\IC$-linear map from the subspace of~$\cE$ spanned by functions as above to $\cG$. 
\end{para}

\begin{para}\label{Par:LaplaceTransformExtended}
We wish to extend the Laplace transform to a linear map defined on the whole of~$\cE$. To do so, we need to define it on the differential ring $\mathcal R \subseteq \cE$ formed by $\IC$-linear combinations 
$$h = \sum_{i=1}^n c_iz^{a_i} \log(z)^{b_i}$$ 
with rational $a_i$ and non-negative integers $b_i$, in a consistent way with \eqref{Eqn:LaplaceTransform} whenever the integral converges. For this, we introduce primitives on $\cR$ by setting
$$S(z^a\log(z)^b)= \begin{cases}
\displaystyle \frac{z^{a+1}}{a+1}\sum_{k=0}^b\frac{b!}{(b-k)!}\frac{(-1)^k}{(a+1)^k}\log(z)^{b-k} & \mbox{ if $a\neq -1$,}\\[8mm]
\displaystyle \frac{\log(z)^{b+1}}{b+1}& \mbox{ if $a = -1$,}
\end{cases}$$ 
and extending $S$ to a linear map $S\colon \mathcal R\to \mathcal R$. We then define a Laplace transform on $\mathcal R$ as
\begin{equation}\label{Eqn:LaplaceTransform2}
\begin{array}{rcl}
\sL\colon \cR & \longrightarrow & \cR\\
h &  \longmapsto & z^n \sL(S^n h), 
\end{array}    
\end{equation}
where $n$ is any sufficiently large integer. The right-hand side of \eqref{Eqn:LaplaceTransform2} is indeed independent of the choice of $n$. In \cite[5.3]{GevreyI}, Andr\'e gives a slightly different definition, which agrees with ours up to polynomials. Direct computation shows that the map $\sL$ satisfies the relations
\begin{eqnarray}
\mathscr L(zh)(z) & = & -\partial \mathscr Lh(z) \label{Eqn:FourierLaplaceRule1bis} + P(z),\\
\mathscr L(\partial h)(z) & = & z\mathscr Lh(z)-h(0)\label{Eqn:FourierLaplaceRule2bis},
\end{eqnarray}
where $P \in \IC[z]$ is a polynomial depending on $h$, and $h(0)$ stands for the coefficient of $z^0\log(z)^0$.
\end{para}

\begin{para}
The combination of \eqref{Eqn:LaplaceTransform} and \eqref{Eqn:LaplaceTransform2} now yields a $\IC$-linear map \hbox{$\sL\colon \cE \to \cG$} by $z$-adic formal completion \cite[5.4]{GevreyI}. Namely, each element of $\cE$ can be written as a formal~sum
\begin{displaymath}
f=\sum_{\substack{a \in \IQ \\ b\in \IZ_{\geq 0}}} c_{a, b} z^a \log(z)^b, 
\end{displaymath} where $c_{a, b}$ are complex numbers, all zero except for finitely many values of $b$ and finitely many classes of $a$ modulo $\IZ$, and one sets 
\begin{equation}\label{eqn:generalLaplace}
    \sL(f)=\sum_{\substack{a \in \IQ \\ b\in \IZ_{\geq 0}}} c_{a, b} \sL(z^a \log(z)^b).
\end{equation} A degree inspection reveals that this map is injective, but not surjective; a $\IC$-linear complement to its image is given by the space of polynomials $\IC[z] \subseteq \cG$, which is clearly a $\IC[z,\partial]$\nobreakdash-submodule. It is thus convenient to view the Laplace transform as a $\IC$-linear bijection taking values in
$$\cG_0 = \cG/\IC[z].$$
As such, our construction of $\sL \colon \cE \to \cG_0$ agrees with Andr\'e's. The inverse of $\sL$ is described explicitly in \cite[2.2]{RivoalFischlerMicro} as a surjective map $\mathscr R_\infty\colon \cG \to \cE$ with kernel $\IC[z]$. Another advantage of taking $\cG_0$ as target of the Laplace transform is that the polynomial terms in \eqref{Eqn:FourierLaplaceRule1bis} and~\eqref{Eqn:FourierLaplaceRule2bis} disappear, as summarised in the following proposition.
\end{para}

\begin{prop}\label{Pro:LaplaceSummary}
The Laplace transform \eqref{eqn:generalLaplace} induces an isomorphism of complex vector spaces \hbox{$\sL\colon \cE \to \cG_0$} satisfying
\begin{equation}\label{Eqn:LaplaceAndFourierWish}
\sL( Lf) = \FT(L) \sL(f)
\end{equation}
for every $f \in \cE$ and every differential operator $L \in \IC[z, \partial]$.
\end{prop}

\begin{para}\label{Par:LaplaceAndFourier}
If the equality \eqref{Eqn:LaplaceAndFourierWish} were to hold in $\cG$ rather than $\cG_0$, then the Laplace transform of an element of $\cE$ annihilated by an operator $L \in \overline\IQ[z,\partial]$ would be a solution in $\cG$ of the operator~$\FT(L)$ and vice versa, and we could have defined $E$-operators simply as those admitting a basis of solutions in the algebra $\cE$. By some miracle, $L$ and $\FT(L)$ would then have the same order. We can of course not just wish the polynomial correction terms away. An example of an operator that admits a basis of solutions in $\cE$ without being an~$E$\nobreakdash-operator is given in~\hbox{\cite[Rem., p.\,721]{GevreyI}}. The operator $L=(z-1)\partial -z$ admits the $E$\nobreakdash-function $E(z)=(z-1)e^z$ as a basis of solutions, but its Fourier transform 
$$\FT(L)= (1-z)\partial - (z+1)$$
is not a $G$-operator since it has an irregular singularity at infinity. A solution of $\FT(L)$ is given by the function $(z-1)^{-2}e^{-z}$, which is clearly not a $G$-series. The Laplace transform of~$E(z)$ is the $G$-series $G(z)= (2-z)(z-1)^{-2}$, which is however not a solution of $\FT(L)$. The reason for this annoying behaviour can be traced back to the constant term $h(0)$ in \eqref{Eqn:FourierLaplaceRule2bis}. The Fourier transform of $zL = (z-1)z\partial -z^2$ annihilates that constant, and hence also the~$G$\nobreakdash-series~$G(z)$, but it annihilates $(z-1)^{-2}e^{-z}$ as well, so $zL$ is still not an $E$-operator.
\end{para}

\vspace{14mm}
\section{\texorpdfstring{Differential modules associated with $E$-operators}{Differential modules associated with E-operators}}

\begin{par}
In this section, we recast the previous results about differential operators in terms of differential modules, and mitigate the problems raised in \ref{Par:LaplaceAndFourier}. To this end, we introduce following Katz \cite{Katz} two $\overline \IQ$-linear tannakian categories $\Conn_0(\IG_m)$ and $\RS_0(\IA^1)$ that are equivalent via Fourier transform. The former contains a category~$\bE$ of~$\sD_{\IG_m}$\nobreakdash-modules associated with~$E$\nobreakdash-functions, and the latter a category~$\bG_0$ of~$\sD_{\IA^1}$\nobreakdash-modules associated with $G$\nobreakdash-series. These categories also play an important role in the study of exponential motives~\cite{FJ}.
\end{par}

\begin{para}
All $\sD$-modules will be defined over the affine line $\IA^1 = \spec\overline\IQ[z]$ or the multiplicative group $\IG_m = \spec\overline\IQ[z,z^{-1}]$. They are thus differential modules over the ring $\overline\IQ[z]$ or~$\overline\IQ[z,z^{-1}]$, or in other words, left modules for one of the Weyl algebras
\begin{equation}\label{Eqn:InclusionOfWeilAlgebras}
\sD_{\IA^1} = \overline\IQ[z, \partial] \quad \subseteq \quad \sD_{\IG_m} = \overline\IQ[z,z^{-1}, \partial].
\end{equation}
We will tacitly suppose that $\sD$-modules are finitely generated holonomic left $\sD$-modules. We call a $\sD$\nobreakdash-module \emph{simple} if it has no submodules other than itself and the trivial module. Associated with the inclusion $j \colon \IG_m \hookrightarrow \IA^1$,  there is a pair of adjoint functors
$$j_\ast\colon \sD_{\IG_m}\mbox{-Mod} \to \sD_{\IA^1}\mbox{-Mod}\qqet j^\ast\colon \sD_{\IA^1}\mbox{-Mod} \to \sD_{\IG_m}\mbox{-Mod},$$
where $j_\ast$ is the transport of structure via the inclusion \eqref{Eqn:InclusionOfWeilAlgebras}, and $j^\ast$ is the localisation in $z$. The \emph{generic fibre} of a $\sD$-module $M$ is the $\overline\IQ(z)[\partial]$-module $\overline\IQ(z) \otimes_{\overline\IQ[z]} M$, on which the action of~$\partial$ on $M$ is extended through Leibniz's rule, and its \emph{rank} is the dimension of the generic fibre as a~$\overline\IQ(z)$\nobreakdash-vector space. The Fourier transform of a $\sD_{\IA^1}$-module $M$ is the~$\sD_{\IA^1}$\nobreakdash-module $\FT(M)$ with the same underlying vector space $M$, on which $\sD_{\IA^1}$ acts via the ring automorphism \eqref{Eqn:FTasRingAutomorhpismOfD}. Given a differential operator $L$, the identity map on $\sD_{\IA^1}$ induces a canonical isomorphism
\begin{equation}\label{Eqn:ModuleAndOperatorFT}
\FT(\sD_{\IA^1}/\sD_{\IA^1} L) \cong \sD_{\IA^1}/\sD_{\IA^1} \FT(L).    
\end{equation}
\end{para}

\begin{para}
Let $\cA$ be a differential $\overline\IQ(z)$-algebra with field of constants $K$. Assume $\cA$ is an integral domain. A \emph{solution} of a differential module $M$ in $\cA$ is a morphism of differential modules $s\colon M \to \cA$. If $M$ is of the form $M=\sD/\sD L $ for some operator $L$, solutions of~$M$ correspond to solutions of the differential equation $Ly=0$ by evaluating at the class of $1$. The set $\Hom_\sD(M,\cA)$ of solutions of $M$ in $\cA$ is a~$K$\nobreakdash-vector space of dimension at most the rank of~$M$, and we say that $M$ admits a \emph{basis of solutions} in $\cA$ if the equality
$$\dim_K\Hom_\sD(M,\cA) = \mathrm{rank}(M)$$
holds. If that is the case, we may choose a $\overline\IQ(z)$-basis $m_1,\dots,m_n$ of the generic fibre of $M$ and a $K$-basis $s_1, \ldots,s_n$ of the space of solutions, and form the square matrix
$$F = (s_i(m_j))_{1\leq i,j\leq n},$$
which is called a \emph{fundamental matrix of solutions} of $M$. The determinant of~$F$ is non-zero, and the subalgebra of the fraction field of $\cA$ spanned by the entries of $F$ and~$\det(F)^{-1}$ is a Picard-Vessiot extension for $M$; its automorphisms are the $K$-points of the differential Galois group of $M$; see \cite[1.4]{SingerVanDerPut}. Given $\sD$-modules $M_1$ and $M_2$, solutions of their sum and tensor product can be understood in terms of solutions of $M_1$ and $M_2$. For instance, the map
\begin{equation}\label{Eqn:SolutionsOfTensorProduct}
\Hom_\sD(M_1, \cA)\otimes_K \Hom_\sD(M_2, \cA) \to \Hom_\sD(M_1\otimes M_2, \cA)    
\end{equation}
that sends $s_1\otimes s_2$ to $m_1\otimes m_2 \mapsto s_1(m_1)s_2(m_2)$ is well defined and injective; if $M_1$ and~$M_2$ admit a basis of solutions in $\cA$, then this map is bijective by dimension count, and thus~$M_1\otimes M_2$ also admits a basis of solutions in $\cA$.
\end{para}

\begin{prop}[Katz]\label{Pro:PhaseStationnaire}
Let $M$ be a $\sD_{\IG_m}$-module such that $\FT(j_\ast M)$ has regular singularities, including at infinity. Then $M$ is non-singular on $\IG_m$, and $\FT(j_\ast M)$ has vanishing de Rham cohomology. Moreover,~$M$ and~$\FT(j_\ast M)$ have the same rank. 
\end{prop}

\begin{proof}
It follows from \cite[Rem.\,12.3.2]{Katz} that $M$ is non-singular on $\IG_m$. Alternatively, one can show by direct computation that if $L \in \overline\IQ[z,\partial]$ is regular singular, then the only singularities of $\FT(L)$ are $0$ and $\infty$; one can deduce the first statement from that. The de Rham cohomology of $\FT(j_\ast M)$ is computed by the complex $\partial\colon  \FT(j_\ast M)\to\FT(j_\ast M)$, but since~$z$ acts bijectively on $j_\ast M$, so does~$\partial$ on $\FT(j_\ast M)$, hence the second statement. Finally, the statement about ranks is \cite[Prop.\,12.4.4]{Katz}.
\end{proof}

\begin{para}\label{Par:ConnAndRS}
After Katz \cite[12.4]{Katz}, consider the category $\RS(\IA^1)$ of regular singular $\sD_{\IA^1}$\nobreakdash-modules and the full subcategory $\RS_0(\IA^1)$ of those with vanishing de Rham cohomology. We write $\Conn(\IG_m)$ for the category of non-singular $\sD_{\IG_m}$-modules, and $\Conn_0(\IG_m)$ for the full subcategory consisting of those modules~$M$ such that $\FT(j_\ast M)$ is regular singular. According to Proposition \ref{Pro:PhaseStationnaire}, the functors
$$\begin{array}{rclcrcl}
\Conn_0(\IG_m) &\longrightarrow & \RS_0(\IA^1) &\qquad\qquad    &  \RS_0(\IA^1) & \longrightarrow & \Conn_0(\IG_m)\\
 M & \mapsto & \FT(j_\ast M) & &A & \mapsto & j^\ast \FT^{-1}(A)  
\end{array}$$
are mutually inverse equivalences of categories. As non-singular $\sD_{\IG_m}$-modules can be identified with vector bundles with connection on $\IG_m$, the category $\Conn(\IG_m)$ with the usual tensor product structure is a $\overline \IQ$-linear neutral tannakian category; a fibre functor is provided by sending a vector bundle with connection to the fibre at $1$ of the underlying vector bundle. The subcategory~$\Conn_0(\IG_m)$ consists of objects that are regular singular at $0$ and whose slopes at $\infty$ lie in $\{0, 1\}$. It is stable under tensor products, duals, and extraction of subquotients, and hence a tannakian subcategory as shown in~\hbox{\cite[Th.\,12.3.6]{Katz}.} The above equivalences of categories exchange the tensor product of vector bundles with connection and the additive convolution, defined as
$$A \ast B = \Sum_\ast( \pr_1^\ast A \otimes \pr_2^\ast B)$$
for objects $A$ and $B$ of $\RS_0(\IA^1)$, where $\Sum, \pr_1, \pr_2\colon \IA^2\to \IA^1$ stand for summation and projections. The category $\RS_0(\IA^1)$ is thus tannakian with respect to additive convolution.
\end{para}

\begin{ex}
\begin{par}
Given an algebraic number $s\in \overline\IQ$, we set 
\[
E(s)=\sD_{\IG_m}/\sD_{\IG_m}z(\partial - s). 
\] Since solutions of $\partial - s$ are scalar multiples of the exponential function~$e^{sz}$, the rank\nobreakdash-one $\sD_{\IG_m}$\nobreakdash-module~$E(s)$ is non-singular at $0$ and has an irregular singularity at~$\infty$ unless~$s=0$, in which case $E(0)$ is equal to~$\cO_{\IG_m}$. The Fourier transform $\FT(j_\ast E(s))$ is the $\sD_{\IA^1}$\nobreakdash-module $\sD_{\IA^1}/\sD_{\IA^1}\partial(z-s)$, which has rank one and regular singularities at $\infty$ and $s$; solutions are scalar multiples of the meromorphic function $(z-s)^{-1}$. There is a canonical isomorphism
\begin{equation}\label{Eqn:ExponentialProduct}
E(s_1) \otimes E(s_2) = E(s_1+s_2)    
\end{equation}
for all $s_1,s_2\in \overline\IQ$. For any additive subgroup $\Lambda\subseteq \overline\IQ$, the family $\{E(s)\}_{s\in \Lambda}$ of objects of~$\Conn_0(\IG_m)$ hence additively generates a tannakian subcategory, which is equivalent to the category of $\Lambda$-graded vector spaces. Twisting modules by $E(s)$ corresponds via Fourier transform to translation by $s$, in the sense that there is a natural isomorphism
\begin{equation}\label{Eqn:TwistAndShift}
    \FT(j_\ast(M\otimes E(s))) = [z-s]^\ast \FT(j_\ast M)
\end{equation}
for every object $M$ of $\Conn_0(\IG_m)$. 
\end{par}

\begin{par}
Another family of one-dimensional objects of $\Conn_0(\IG_m)$ is given by the \emph{Kummer modules}
\[
K(a) = \sD_{\IG_m}/\sD_{\IG_m}(z\partial - a)
\] for $a\in \overline\IQ$. Solutions of $z\partial-a$ are scalar multiples of $z^a$, and $K(a)$ is isomorphic to $K(a+1)$ through multiplication by~$z$. The module~$K(a)$ has regular singularities at $0$ and $\infty$, and its Fourier transform is given by
\[
\FT(j_\ast K(a)) = \sD_{\IA^1}/\sD_{\IA^1}(z\partial+a).
\] In fact, every one-dimensional object of $\Conn_0(\IG_m)$ is isomorphic to $E(s)\otimes K(a)$ for some unique~$s\in \overline\IQ$ and $a\in \overline\IQ/\IZ$. This follows from the fact that giving a connection on a rank-one vector bundle on $\IG_m$ amounts to giving a differential form $\omega \in \overline{\IQ}[z, z^{-1}]dz$ and that having a regular singularity at $0$ and slopes in $\{0, 1\}$ at~$\infty$ implies that $\omega$ is of the form $a\sfrac{dz}{z}+sdz$.  
\end{par}
\end{ex}

\begin{para}\label{Par:IntroDivisorOfM}
\begin{par}
As this will become important later, we now show a remarkable property of the tannakian category $\Conn_0(\IG_m)$, namely, the existence of a monoidal functor
\begin{equation}\label{Eqn:StokesDecomposition}
\Psi \colon \Conn_0(\IG_m) \longrightarrow \{\mbox{$\overline \IQ$-graded vector spaces}\}
\end{equation}
that yields a fibre functor after forgetting the grading. To construct it, we start by localising at infinity. Let $M$ be an object of $\Conn_0(\IG_m)$, and define~$M_\infty$ as the $\overline \IQ(\!(w)\!)$-differential module
$$M_\infty = \overline \IQ(\!(w)\!) \otimes_{\overline \IQ[z]} M,$$
where the algebra structure $\overline \IQ[z] \to \overline \IQ(\!(w)\!)$ is given by sending $z$ to $w^{-1}$. By \hbox{\cite[Prop. 2.3.4]{KatzGalois},} there is a natural decomposition of $M_\infty$ into a regular singular part and a complementary part of slope $1$. We obtain from this decomposition a natural morphism
\begin{equation}\label{Eqn:InfinityStokesDecomposition}
\bigoplus_{s\in \overline \IQ} M_\infty^{(s)} \otimes E(s)_\infty \longrightarrow M_\infty,
\end{equation}
where $M_\infty^{(s)}$ denotes the regular singular part of $(M\otimes E(-s))_\infty$. Before we continue with the construction of the functor $\Psi$, we verify the following:
\end{par}
\end{para}

\begin{lem}\label{Lem:StokesDecompositionCheck}
The natural morphism of differential $\overline\IQ(\!(w)\!)$-modules \eqref{Eqn:InfinityStokesDecomposition} is an isomorphism, which is compatible with tensor products and duals.
\end{lem}

\begin{proof}
In order to check that \eqref{Eqn:InfinityStokesDecomposition} is an isomorphism, it suffices to show that for every irreducible differential $\overline \IQ(\!(w)\!)$-module $M$ of slope $0$ or $1$, there exists a unique $s\in \overline \IQ$ such that~$M \otimes E(-s)_\infty$ is of slope $0$. The existence of such an $s$ implies surjectivity, and injectivity follows from its uniqueness. Compatibility with tensor products and duals is obvious from~\eqref{Eqn:ExponentialProduct} and the fact that $E(s)$ has slope $0$ if and only if $s=0$. Thus, let $M$ be an irreducible~$\overline \IQ(\!(w)\!)$-module of slope $0$ or $1$. The Levelt--Turrittin decomposition \cite[2.2]{KatzGalois} shows that $M$ has rank one and is given by a differential operator $L = w\partial_w - f$ with 
$$f(w) = \frac t w + a + wh(w)$$
for some $t,a \in \overline \IQ$ and $h \in \overline\IQ[\![w]\!]$. By definition, the slope of $M$ is equal to $\max(-\mathrm{ord}_w(f),0)$, so $M$ is of slope $0$ if and only if $t = 0$. A non-zero solution of $L$ is given by $e^{-t/w} \cdot w^a \cdot H(w),$ where~$H$ is the exponential of a primitive of $h$, whereas a non-zero solution of $E(-s)_\infty$ is given by the function $e^{-sz}= e^{-s/w}$. A non-zero solution of $M \otimes E(-s)_\infty$ is hence $e^{-(s+t)/w} \cdot w^a \cdot H(w)$, and our claim follows on noting that this function belongs to the algebra of convergent Laurent series with monodromy in the variable $w$ if and only if $s=-t$.
\end{proof}

\begin{para}\label{Para:FibreFunctor}
\begin{par}
We now have obtained a natural decomposition of $M_\infty$, compatible with tensor products and duals. Besides, the category of regular singular $\overline \IQ(\!(w)\!)$-differential modules is equivalent, as a tannakian category, to the category of $\overline\IQ$-vector spaces equipped with an automorphism. Choose a fiber functor $\omega$ on it, and then set
\begin{equation}\label{Eqn:DefinitionStokesDecomposition}
\Psi_s(M) =  \omega\big(M_\infty^{(s)}\big) \qqet \Psi(M) = \bigoplus_{s\in \overline \IQ} \Psi_s(M) .
\end{equation}
By forgetting the grading, $\Psi$ yields a fibre functor on $\Conn_0(\IG_m)$. A remarkable consequence of this is that the Galois group of every object $M$ with respect to $\Psi$ canonically contains the torus with character group the subgroup of $\overline \IQ$ generated by those $s\in \overline \IQ$ with $\Psi_s(M)\neq 0$. 
\end{par}

\begin{par}
We write $\IZ[\overline \IQ]$ for the group ring of the additive group $\overline \IQ$, that is, the free $\IZ$-module generated by symbols $[s]$, one for each $s\in \overline \IQ$, together with the multiplication uniquely determined by distributivity and $[s_1]\ast [s_2] = [s_1+s_2]$ for all $s_1,s_2\in \overline \IQ$. We define group homomorphisms called \emph{degree}, \emph{inversion}, and \emph{evaluation} on generators as follows: 
$$\begin{array}{rclrclrcl}
\deg\colon \IZ[\overline \IQ] & \longrightarrow & \IZ, &\qquad \inv\colon \IZ[\overline \IQ] & \longrightarrow & \IZ[\overline \IQ], & \qquad \ev\colon \IZ[\overline \IQ] & \longrightarrow & \overline \IQ. \\
\relax [s] & \mapsto & 1 & [s] & \mapsto & [-s] & [s] & \mapsto & s
\end{array}$$
Note that the degree and the inversion maps are ring homomorphisms as well. We call 
$$\div(M) = \sum_{s\in \overline \IQ} \dim(\Psi_s(M))[s]$$
the \emph{divisor} of the object $M$ of $\Conn_0(\IG_m)$. Since $\Psi$ is a fibre functor, and hence preserves dimensions, the equality $\dim M = \deg(\div M)$ holds. Exactness and compatibility with tensor products of $\Psi$ implies the relations
\begin{equation*}
\div(M) = \div(M_1) + \div(M_2) \qqet \div(M_1\otimes M_2) = \div(M_1)\ast \div(M_2)
\end{equation*}
for every short exact sequence $0 \to M_1 \to M \to M_2 \to 0$ and all objects $M_1$ and $M_2$ of~$\Conn_0(\IG_m)$ respectively. The compatibility of $\Psi$ with duals amounts to the relation \hbox{$\div(M^\vee) = \inv(\div(M))$}, and $\div(\det (M)) = [\ev(\div(M))]$ holds.
\end{par}
\end{para}

\begin{para}
We can express the divisor of an object of $\Conn_0(\IG_m)$ in terms of its Fourier transform as follows. Given an object $A$ of $\RS(\IA^1)$, let $\Phi_{z-s}(A)$ denote the vector space of vanishing cycles of $A$ with respect to the function $z-s$. If $A$ is torsion free, this is the space of solutions of $[z-s]^\ast A$ in $\cM$ modulo the space of holomorphic solutions. We set
$$\sing(A) = \sum_{s\in \overline\IQ} \dim(\Phi_{z-s}(A))[s].$$
If $A$ is of the form $\FT(j_\ast M)$ for some object $M$ of 
$\Conn_0(\IG_m)$, the dimension of the regular singular part of $M_\infty$ equals the dimension of $\Phi_{z-0}(A)$ by \hbox{\cite[Cor.\,2.11.7]{Katz}, and hence the equality } 
\begin{equation}\label{Eqn:SingAndDiv}
\sing(A) = \div(M)    
\end{equation}
follows from \eqref{Eqn:TwistAndShift}. We can combine the functors $\Phi_{z-s}$ to form a \emph{total vanishing cycles functor} $\Phi$ from $\RS_0(\IA^1)$ to $\overline\IQ$-graded vector spaces.
As one might guess, there exist natural isomorphisms $\Phi(A\ast B)\cong \Phi(A)\otimes \Phi(B)$ and $\Phi(A^\vee)=\Phi(A)^\vee$, which may be interpreted as some type of global Thom-Sebastiani isomorphism. We found it harder than expected to write down such isomorphisms and verify their compatibility with associativity and commutativity constraints. Fortunately, the numerical identity \eqref{Eqn:SingAndDiv} is all we need.
\end{para}

\begin{defn}\label{defn:modulestypeGandE}
A $\sD_{\IA^1}$-module $M$ is said to be of \emph{type $G$} if its generic fibre is of the form $\overline\IQ(z)[\partial]\slash \overline\IQ(z)[\partial]L$ for some $G$-operator~$L$. A $\sD_{\IG_m}$-module $M$ is said to be of \emph{type $E$} if its Fourier transform~$\FT(j_\ast M)$ is of type $G$. We write 
$$\bG \subseteq \RS(\IA^1) \qqet \bE \subseteq \Conn_0(\IG_m)$$
for the full subcategories of $\sD_{\IA^1}$-modules of type $G$ and $\sD_{\IG_m}$-modules of type $E$ respectively, and $\bG_0 \subseteq \bG$ for the full subcategory consisting of those modules with vanishing cohomology.
\end{defn}

\begin{para} The definition of $\sD_{\IA^1}$-modules of type $G$ is due to André \cite[3.6]{GevreyI}. In the same paper, modules of type $E$ are defined on $\IA^1$ as those whose Fourier transform is of type $G$; see~\hbox{\cite[4.9]{GevreyI}}. Our definition and his are compatible in that the functors $j_\ast$ and $j^\ast$ preserve $\sD$-modules of type $E$. Indeed, this is obviously true for $j_\ast$, and follows from Theorem \ref{Thm:PropertiesOfGModules} below for $j^\ast$. There are, however, many $\sD_{\IA^1}$-modules of type $E$ that are not of the form $j_\ast M$. We prefer to work with modules on $\IG_m$ in order to ensure that the equivalent characterisations of being of type $E$ from Proposition \ref{Pro:WishListC} below hold.
\end{para}

\begin{ex}\label{Exmp:ExponentialPeriodModules}
The exponential modules $E(s)$ for $s\in \overline \IQ$ and the Kummer modules~$K(a)$ for $a\in \IQ$ belong to $\bE$ and have divisors 
$$\div(E(s))=[s] \qqet \div(K(a))=[0].$$
To give a more interesting example, let $f\in \overline \IQ[t]$ be a polynomial of degree $n \geq 2$, which we regard as a morphism $\IA^1 \to \IA^1$. The $\sD$-module $f_\ast\cO_{\IA^1}$ is the $\IQ[z]$\nobreakdash-module $\overline\IQ[z,u]/(f(u)-z)$, equipped with the unique derivation extending $\partial$ on~$\IQ[z]$. As a $\overline\IQ[z]$-module, it is free of rank~$n$, a basis being given by the classes of~$1,u,u^2,\ldots,u^{n-1}$. A basis of solutions of $f_\ast \cO_{\IA^1}$ is given by the set of $\overline\IQ[z]$-algebra morphisms $\overline\IQ[z,u]/(f(u)-z) \to \cG$ sending $u$ to an algebraic function~$u(z)$ satisfying $f(u(z))=z$, which indeed can be found in $\cG$ since $f$ has algebraic coefficients. The module $f_\ast\cO_{\IA^1}$ is thus of type $G$. The adjunction map $\cO_{\IA^1} \to f_\ast\cO_{\IA^1}$ sends~$1$ to the class of $1$ and induces an isomorphism in cohomology, so $A = (f_\ast\cO_{\IA^1})/\cO_{\IA^1}$ belongs to~$\bG_0$, and~$M=j^\ast \FT^{-1}(A)$ belongs to $\bE$. Moreover, we can retrieve $A$ from $M$ by means of the isomorphism~$A\cong\FT(j_\ast M)$. The divisor of $M$ is given by the critical values of~$f$, namely
$$\div(M)=\sing (A)=\sum_{f'(\alpha)=0}[f(\alpha)],$$ 
where the sum runs over all zeroes $\alpha$ of $f'$, counted with multiplicity.
\end{ex}

\begin{thm}[André]\label{Thm:PropertiesOfGModules} Modules of type $G$ satisfy the following: 
\begin{enumerate}
    \item A $\sD_{\IA^1}$-module is of type $G$ if and only if it admits a basis of solutions in $\cG$.
    \item The class of modules of type $G$ is stable under extensions, tensor product, and duals, and every submodule and quotient of a module of type $G$ is again of type $G$.
\end{enumerate}
\end{thm}

\begin{proof}
\begin{par}
By Theorem \ref{Thm:PropertiesOfGoperators}, $G$-operators are regular singular, thus the inclusion $\bG \subseteq \RS(\IA^1)$. Since the generic fibre of a $\sD_{\IA^1}$-module $A$ is isomorphic to $\overline \IQ(z)[\partial]\slash \overline \IQ(z)[\partial]L$ for some operator~$L$, the module $A$ admits a basis of solutions in $\cG$ if and only if the operator $L$ does, so statement (1) follows directly from the definitions.
\end{par}

\begin{par}
That modules of type $G$ are stable under extensions, extraction of subquotients, and duality is shown in \cite[IV\,3, Lem.\,2]{Gfunctions}. Were it not a circular argument, we could of course deduce this from statement (2) of Theorem \ref{Thm:PropertiesOfGoperators} and the observation that any short exact sequence of~$\overline\IQ(z)[\partial]$-modules is isomorphic to an exact sequence associated with a product of differential operators. It remains to show that if $A_1$ and $A_2$ are of type $G$, then so is $A_1\otimes A_2$. This follows from $(1)$ and the general fact that solutions of $A_1\otimes A_2$ can be expressed in terms of products of solutions of $A_1$ and $A_2$, as shown in \eqref{Eqn:SolutionsOfTensorProduct}.
\end{par}
\end{proof}

\begin{prop}\label{Pro:WishListC}
Let $M$ be a $\sD_{\IG_m}$-module. The following statements are equivalent: 
\begin{enumerate}
    \item The module $M$ is of type $E$, that is, the module $A = \FT(j_\ast M)$ is of type $G$.
    \item The module $M$ belongs to $\Conn_0(\IG_m)$ and admits a basis of solutions in $\cE$.
    \item There exists an $E$-operator $L\in \overline \IQ[z, \partial]$ such that $M$ is isomorphic to $\sD_{\IG_m}/ \sD_{\IG_m}L$.
\end{enumerate}
\end{prop}

\begin{proof}
\begin{par}
$(1) \Longleftrightarrow (2)$ Let $M$ be a $\sD_{\IG_m}$-module, and set $A = \FT(j_\ast M)$. To say that $M$ is of type $E$ is to say that $A$ admits a basis of solutions in $\cG$. Since modules of type $G$ are regular singular, taking Proposition \ref{Pro:PhaseStationnaire} into account, all that remains to prove is that $M$ admits a basis of solutions in $\cE$ if and only if $A$ admits a basis of solutions in $\cG$. To do so, we will use the Laplace transform $\sL\colon \cE \to \cG_0$ as described in Proposition~\ref{Pro:LaplaceSummary}. The vector space underlying~$A$ is the same as that of $M$, and we write $a=\widehat{m} \in A$ for the clone of~$m \in M$. For every $L\in \overline \IQ[z,\partial]$, the identity $\widehat{Lm} = \FT(L)\widehat m$ holds by definition. By requiring the diagram 
$$\begin{diagram}\setlength{\dgARROWLENGTH}{4mm}
\node{M}\arrow{s,l}{\widehat-}\arrow[2]{e,t}{s}\node[2]{\cE}\arrow{s,r}{\mathscr L}\\
\node{A}\arrow[2]{e,t}{t}\node[2]{\cG_0}
\end{diagram} \qquad\qquad \begin{array}{rclrcl} \widehat{\partial m}  & = & -z\widehat m, \qquad & \mathscr L (\partial f) & =& - z\mathscr L(f), \\ \widehat{z m} & = & -\partial \widehat m, & \mathscr L (z f) & = & -\partial \mathscr L(f) \end{array}$$
to commute, every $\overline \IQ[z, \partial]$-linear map $s\colon M\to \cE$ determines a $\overline \IQ[z, \partial]$-linear map \hbox{$t\colon A \to \cG_0$} and vice versa. Since the operator $\partial$ acts bijectively on $A$ and nilpotently on $\IC[z]$, the vanishing~$\Hom_\sD(A,\IC[z])=\Ext^1_\sD(A,\IC[z])=0$ holds. Hence, the exact sequence
$$0 \to \Hom_\sD(A, \IC[z]) \to \Hom_\sD(A, \cG) \xrightarrow{\:\:\cong\:\:} \Hom_\sD(A, \cG_0) \to \Ext^1_\sD(A,\IC[z]) \to \cdots$$
degenerates to an isomorphism as indicated and we obtain a linear bijection
$$\Hom_{\sD}(M, \cE) \cong \Hom_{\sD}(A, \cG).$$
The modules $A$ and $M$ have the same rank by Proposition \ref{Pro:PhaseStationnaire}, and hence $A$ admits a basis of solutions in $\cG$ if and only if $M$ admits a basis of solutions in $\cE$. 
\end{par}

\begin{par}
$(1) \Longrightarrow (3)$ Being non-singular on $\IG_m$, the module $M$ is isomorphic to $\sD_{\IG_m}/\sD_{\IG_m}L$ for some operator $L\in \overline\IQ[z,\partial]$. We claim that $L$ is an $E$-operator. Indeed, letting $\delta_0 = \sD_{\IA^1}/\sD_{\IA^1} z$ denote the Dirac module supported at $0$, there is an exact sequence of $\sD_{\IA^1}$-modules 
$$0 \to \delta_0^a \to \sD_{\IA^1}/\sD_{\IA^1} L \to j_\ast M \to  \delta_0^b \to 0$$
for some integers $a,b\geq 0$, which after Fourier transform yields an exact sequence
$$0 \to \cO_{\IA^1}^a \to \sD_{\IA^1}/\sD_{\IA^1} \FT(L) \to A \to \cO_{\IA^1}^b\to 0.$$
By assumption, $A$ is of type $G$, as well as $\cO_{\IA^1}$, and hence $\sD_{\IA^1}/\sD_{\IA^1} \FT(L)$ is of type $G$ by part~(2) of Theorem~\ref{Thm:PropertiesOfGModules}. This means that $\FT(L)$ is a $G$-operator, and thus $L$ an $E$-operator.
\end{par}

\begin{par}
$(3) \Longrightarrow (2)$ Given an $E$-operator $L \in \overline\IQ[z,\partial]$, set $M = \sD_{\IG_m}/\sD_{\IG_m}L$ and \hbox{$M_0 = \sD_{\IA^1}/\sD_{\IA^1} L$}, so that $M=j^\ast M_0$. By Theorem \ref{Thm:PropertiesOfEoperators}, the operator $L$, and hence the module $M$, admits a basis of solutions in $\cE$. It remains to check that $\FT(j_\ast M)$ is regular singular. By assumption, the module $\FT(M_0)$ is of type $G$, and hence regular singular. The kernel and the cokernel of the adjunction map $M_0 \to j_\ast M$ are torsion modules supported at $0$, and hence the kernel and the cokernel of $\FT(M_0) \to \FT(j_\ast M)$
are powers of $\cO_{\IA^1}$. Since submodules, quotients, and extensions of regular singular modules are regular singular, $\FT(j_\ast M)$ is regular singular. 
\end{par}
\end{proof}

\begin{thm}\label{Thm:TannakianSubcategoriesRHEConn}
The category $\bE$ is an abelian subcategory of $\Conn_0(\IG_m)$, which is stable under extensions, tensor product, duality, and extraction of subquotients. In other words, it is a tannakian subcategory.
\end{thm}

\begin{proof}
Stability of $\bE \subseteq \Conn_0(\IG_m)$ under extensions and extraction of subquotients follows from the corresponding statements for the category $\bG$, which we deduce from Theorem \ref{Thm:PropertiesOfGModules}. The characterisation (2) in Proposition \ref{Pro:WishListC} combined with \eqref{Eqn:SolutionsOfTensorProduct} shows that modules of type~$E$ are stable under tensor product. Finally, stability of $\bE$ under duality follows from Theorem~\ref{Thm:PropertiesOfGModules} and compatibility of the Fourier transform with duals, in the sense that 
$$M^\vee = j^\ast \FT^{-1}([-z]^\ast A^\vee)$$
holds for every object $M = j^\ast\FT^{-1}(A)$ of $\Conn_0(\IG_m)$.
\end{proof}

\begin{cor}
The category $\bG_0$ is stable under additive convolution.
\end{cor}

\begin{proof}
The equivalence of categories $\FT(j_\ast -)\colon \bE\to \bG_0$ carries the tensor product in $\bE$ to additive convolution in $\bG_0$. 
\end{proof}

\vspace{14mm}
\section{\texorpdfstring{Hypergeometric $E$-functions and their associated $\sD$-modules}{Hypergeometric E-functions and D-modules}}

\begin{par}
In this section, we introduce a category of $\sD_{\IG_m}$-modules $\bH \subseteq \bE$ whose solutions contain all polynomial expressions in hypergeometric $E$\nobreakdash-functions. Due to the historic nature of the subject, several competing conventions for hypergeometric functions and their differential equations can be found in the literature. To avoid confusion, we start by briefly reviewing~them. 
\end{par}

\begin{para}
Throughout, we adopt the notation $\vartheta = z\partial$. For each integer $n \geq 0$, the $n$-th \emph{rising Pochhammer symbol} of a complex number $x$ is defined as
$$(x)_n = x(x+1)(x+2)\cdots (x+n-1),$$
with the convention $(x)_0=1$. Given integers~$0 \leq p \leq q$ and rational numbers $a_1, \ldots,a_p \in \IQ$ and $b_1, \ldots, b_q \in \IQ\setminus \IZ_{\leq 0}$, we call the formal power series
\begin{equation}\label{Eqn:DefHypergeometricpFq}
F\biggl(\begin{matrix} a_1, \ \ldots, \ a_p \\ b_1, \ \ldots, \ b_q \end{matrix} \:\bigg|\:\: z \biggr)=\sum_{n=0}^\infty \frac{(a_1)_n(a_2)_n\cdots (a_p)_n}{(b_1)_n(b_2)_n\cdots (b_q)_n} \: z^n
\end{equation}
a \emph{hypergeometric function} of type $(p, q)$. The function classically denoted by $_{p}F_{q-1}$ corresponds to the case where the last parameter $b_q$ is equal to $1$, and hence an $n!$ appears in the denominator. The radius of convergence of this series is equal to $1$ if $p=q$, and infinite if~$p<q$. Given non-zero polynomials $P, Q \in \overline\IQ[t]$ with rational roots of degrees $p$ and $q$ and leading coefficients $\lambda_P$ and $\lambda_Q$ respectively, we call
\begin{equation}\label{Eqn:DefHypPQOperator}
\Hyp(P, Q)= Q(\vartheta)-zP(\vartheta)    
\end{equation}
a \emph{hypergeometric differential operator} of type\footnote{Note that the roles of $P$ and $Q$ seem to be reversed with respect to Katz's notation in~\cite[Ch.\,3]{Katz}. Behind this reversal is the change of variables $z\mapsto z^{-1}$, which transforms $\Hyp(P(t), Q(t))$ into $-z^{-1}\Hyp(Q(-t), P(-t))$ and ensures that the irregular singularity is at infinity if $p<q$.} $(p,q)$. The condition $p\leq q$ implies that the operator $\Hyp(P, Q)$ has order $q$. Its singularities are $\{0, \infty\}$ if $p<q$, and $\{0, \lambda_Q/\lambda_P, \infty\}$ if $p=q$. In both cases,~$z=0$ is a regular singularity with indicial polynomial $Q$, and hence every hypergeometric differential operator as above admits a basis of solutions in the differential algebra $\cM$. The adjoint of a hypergeometric operator is again hypergeometric, as it is given by the formula (see \cite[(3.1)]{Katz})
\begin{equation}\label{Eqn:AdjHyp}
\Hyp(P(t),Q(t))^\ast = \Hyp(P(-2-t),Q(-1-t)).
\end{equation}
\end{para}

\begin{para}\label{sec:solutionshyper}
Let $L=\Hyp(P,Q)$ be a hypergeometric operator of type $(p,q)$. For the sake of completeness, we recall how to produce a basis of solutions of the equation $Lu=0$ using Frobenius's method; see, for example, \cite[Ch.\,16]{Ince}. The procedure starts with an expression
\begin{equation}\label{Eqn:FrobeniusAnsatz}
w(z,t) = z^t \cdot \sum_{n=0}^\infty c_n(t)z^n,    
\end{equation}
where $t$ is an auxiliary complex parameter and $c_n$ a rational function of $t$. While~$c_0(t)$ is left undetermined for the time being, we set
$$c_n(t) = \frac{P(t+n-1)}{Q(t+n)}c_{n-1}(t) = c_0(t)\prod_{k=0}^{n-1}\frac{P(t+k)}{Q(t+k+1)}$$
for each $n\geq 1$. With this definition, applying $L$ to $w$ yields
$$Lw(z,t) = z^t\cdot \sum_{n=0}^\infty c_n(t)\Big(Q(t+n) - zP(t+n)\Big)z^n = z^t\cdot c_0(t)Q(t).$$
Thus, if $\beta$ is a root of the indicial polynomial $Q$ satisfying $Q(\beta+n)\neq 0$ for all integers $n\geq 1$, then we can set $c_0(t)=1$, specialise $t=\beta$, and obtain the solution
\begin{equation}\label{Eqn:HypergeometricSolution}
z^\beta \cdot F\biggl(\begin{matrix} \beta-\alpha_1, & \ldots, & \beta- \alpha_p \\ \beta-\beta_1+1, & \ldots, & \beta-\beta_q+1 \end{matrix} \:\bigg|\:\: \lambda z \biggr),   
\end{equation}
where $P(t) = \lambda_P(t-\alpha_1)\cdots (t-\alpha_p)$ and $Q(t) = \lambda_Q(t-\beta_1)\cdots (t-\beta_q)$ and $\lambda = \lambda_P/\lambda_Q$. In the \emph{non-resonant} case, that means if the roots of $Q$ are simple and distinct modulo $\IZ$, a basis of solutions of~$L$ is obtained this way. In general, we organise the roots of $Q$, always counted with multiplicity, into congruence classes modulo $\IZ$. After ordering the $r+1$ roots in a chosen class, say $\beta_0, \beta_1, \ldots, \beta_r$, so that $\Re(\beta_r) \leq \cdots \leq \Re(\beta_0)$ holds, we set $d= \beta_0-\beta_r$ and
$$c_0(t) = Q(t+1)Q(t+2) \cdots Q(t+d).$$
With this choice of $c_0$, the rational function $c_n$ is regular at $\beta_k$ for all $n\geq 0$ and $r\geq k \geq 0$. Moreover, $\beta_k$ is a root of $c_0(t)Q(t)$ of order $\geq k+1$. It follows that the functions
$$\big(\partial^k_t w(z,t)\big)\Big|_{t=\beta_k} \qquad \qquad k=0,1,2,\ldots,r$$
are solutions of $L$. Their linear independence is straightforward to check, since in the expression 
$$\big(\partial^k_t w(z,t)\big)\Big|_{t=\beta_k} = \sum_{j=0}^k  z^{\beta_k}\log(z)^{k-j}\cdot  \binom{k}{j} \sum_{n=0}^\infty c_n^{(j)}(\beta_k)z^n $$
the term with the highest power of the logarithm is, up to a non-zero scalar, $z^{\beta_k} \log(z)^k F(\lambda z)$, where $F$ is the hypergeometric series given in \eqref{Eqn:HypergeometricSolution}. It is worth noting that
$$\sum_{n=0}^\infty c_n^{(j)}(\beta_k)z^n$$
is in general \emph{not} a hypergeometric series. Its coefficients resemble polynomial expressions of harmonic numbers in various guises; see \cite[(2.1), (2.4), (2.5)]{Vovkodav}. It is a natural question whether these series are polynomial expressions in hypergeometric functions. For example, for $P=-1$ and $Q=t^2/4$, the resulting series $\sum_{n=1}^\infty \big(\sum_{m=1}^n \tfrac 1m\big) \frac{z^n}{n!}$ is indeed a product of hypergeometric functions as shown in \cite[p.\,2]{RivoalFischler}. It seems however doubtful that this is always the case. 
\end{para}

\begin{thm}\label{Thm:RegularHypergeimetricIsG}
Let $P, Q \in \overline \IQ[t]$ be non-zero polynomials with rational roots. If $P$ and $Q$ have the same degree, then the hypergeometric operator $\Hyp(P,Q)$ is a $G$-operator.
\end{thm}

\begin{proof}
This is well-known and can be proved in several ways. A first possibility is to show that~$\Hyp(P,Q)$ comes from geometry, in the sense that it is a factor of some Picard\nobreakdash-Fuchs differential operator over $\overline\IQ$, as is done for example in \cite[I\,4.4, p.\,31]{Gfunctions} or \cite[Th.\,5.4.4]{Katz}. The theorem then follows from the fact that those are $G$-operators, which is proved in \cite[V\,App.]{Gfunctions}. A more elementary approach is to check that all solutions constructed in~\eqref{sec:solutionshyper} belong to the algebra~$\cG$, which in the non-resonant case amounts to checking that hypergeometric series with an equal number~$p=q$ of rational parameters are $G$\nobreakdash-series; see, for example,~\hbox{\cite[I\,4.4, p.\,24]{Gfunctions}.} A third possibility is to use that factors of hypergeometric differential operators are either hypergeometric with the same $q-p$ or Kummer operators to reduce to the case where $\Hyp(P,Q)$ is irreducible, and then deduce the result from Chudnovsky's theorem stated in~\ref{Par:DefinitionsOfGoperator}.
\end{proof}

\begin{para}\label{Par:HypergeometricEFunctions}
Let $0\leq p<q$ be integers. Siegel showed that for any choice of $a_1, \ldots, a_p \in \IQ$ and $b_1, \ldots,b_q \in \IQ\setminus \IZ_{\leq 0}$ and $\lambda \in \overline\IQ$, the function
\begin{equation}\label{Eqn:DefHypergeometricEfct}
F\biggl(\begin{matrix} a_1, \ \ldots, \ a_p \\ b_1, \ \ldots, \ b_q \end{matrix} \:\bigg|\:\: \lambda z^{q-p} \biggr)
\end{equation}
is an $E$\nobreakdash-function; see \cite[p.\,224]{Siegel1929} and \cite[II.9]{Siegel}. Following him, we call it a \emph{hypergeometric $E$\nobreakdash-function}. The series~\eqref{Eqn:DefHypergeometricpFq} also makes  sense for algebraic or even complex parameters, but it is then only in very special cases an $E$\nobreakdash-function, and in that case always a $\overline\IQ$-linear combination of hypergeometric series with rational parameters; the precise conditions have been worked out by Galoshkin \cite{galoshkin}.
\end{para}

\begin{para}
By changing the variable $z$ in the classical hypergeometric differential operator to~$z^{q-p}$, we obtain an operator annihilating \eqref{Eqn:DefHypergeometricEfct}. The main theorem \ref{Thm:EHypergeometricOperators} of this section implies that such a modified operator admits a basis of solutions in the algebra $\cE$. We shall use the following standard notation for Kummer pullback and Kummer induction of $\sD_{\IG_m}$\nobreakdash-modules. Given a non-zero integer $m$, we let $[m] \colon \IG_m \to \IG_m$ denote the étale cover $z \mapsto z^m$ and~$[m]^\ast$ the induced \emph{Kummer pullback} map on $\sD_{\IG_m}$-modules. For example, if $M$ corresponds to a differential operator $L(z, \vartheta)$, then $[m]^\ast M$ corresponds to the operator $L(z^m, \frac{1}{m}\vartheta)$. The functor $[m]^\ast$ has a right adjoint $[m]_\ast$ called \emph{Kummer induction}, and we say that a $\sD_{\IG_m}$-module is \emph{Kummer induced} if it is of the form $[m]_\ast M$ for some $\sD_{\IG_m}$-module $M$ and some $m\geq 2$. 
\end{para}

\begin{defn}\label{Defn:E-hypergeometricmodule}
Let $0 \leq p <q$ be integers, and set $m=q-p$. Given non-zero polynomials with rational roots~$P,Q \in \overline \IQ[t]$ of degrees $p$ and $q$ respectively, we call
$$H(P,Q) = \sD_{\IG_m}/\sD_{\IG_m} \Hyp(P,Q) \qqet H^E(P,Q)=[m]^\ast H(P,Q)$$
a \emph{hypergeometric $\sD_{\IG_m}$-module} and  an $E$-\emph{hypergeometric $\sD_{\IG_m}$-module} of \emph{type} $(p, q)$ respectively. We shall also refer to 
\[
[m]^\ast \Hyp(P,Q)=Q(\tfrac 1m \vartheta)-z^m P(\tfrac 1m \vartheta)
\] as an \emph{$E$-hypergeometric operator}.
\end{defn}

\begin{thm}\label{Thm:EHypergeometricOperators}
Let $0 \leq p <q$ be integers, and set $m=q-p$. Let $P, Q \in \overline \IQ[t]$ be non\nobreakdash-zero polynomials with rational roots of degrees $p$ and $q$ and leading coefficients $\lambda_P$ and $\lambda_Q$ respectively. The $\sD_{\IG_m}$-module $M= H^E(P,Q)$ is of type $E$ and has divisor 
$$\div(M) = p[0] + \sum_{k=0}^{m-1}[m\cdot \rho \cdot e^{2\pi i k/m}],$$
where $\rho$ is an $m$-th root of $\lambda_P/\lambda_Q$.
\end{thm}

\begin{proof}
\begin{par}
The essential part of the proof is a calculation due to Katz; see \cite[proof of Th.\,6.2.1]{Katz}. We set $R(t)= (mt+1)(mt+2)\cdots (mt+m)P(t)$, define operators $L$ and $K$ as
$$L = [m]^\ast\Hyp(P,Q) \qqet K  = \Hyp\big(Q(-t-\tfrac 1m), R(-t-\tfrac 1m)\big),$$
and consider the $\sD_{\IA}$-modules
$$M_0 = \sD_{\IA^1}/\sD_{\IA^1}L, \qquad M_1 = \sD_{\IA^1}/\sD_{\IA^1}\partial^mL, \qqet A_0 = \sD_{\IA^1}/\sD_{\IA^1}K.$$
The $E$-hypergeometric module $M$ is given by $M = j^\ast M_0$. Since $K$ is a  hypergeometric operator of type $(q,q)$, it follows from Theorem \ref{Thm:RegularHypergeimetricIsG} that $A_0$ is of type $G$. Its singularities are located at $0$, at $m^m\lambda_P/\lambda_Q$, and at $\infty$. By statement (4) of Theorem \ref{Thm:PropertiesOfGoperators}, the module~$[m]^\ast A_0$ is of type $G$ as well. We claim that there is an isomorphism of $\sD_{\IA^1}$-modules
\begin{equation}\label{Eqn:FTofKummerPullback}
\FT(M_1) \cong [m]^\ast  A_0.    
\end{equation}
To show this, recall that the Fourier transform of an operator is obtained from it by replacing~$\partial$ with $z$ and $z$ with $-\partial$, and thus $\vartheta = z\partial$ with $-\partial z = -\vartheta-1$. Hence, 
\begin{eqnarray*}
\FT(\partial^mL) & = & \FT(\partial^m Q(\tfrac 1m\vartheta)-\partial^mz^m P(\tfrac 1m \vartheta))\\
& = & z^m Q\big(-\tfrac 1m\vartheta -\tfrac 1m\big)-(-1)^{m} z^m\partial^m P\big(-\tfrac 1m\vartheta-\tfrac 1m\big) \\
&= & z^m Q\big(-\tfrac 1m\vartheta -\tfrac 1m\big)-R\big(-\tfrac 1m\vartheta-\tfrac 1m\big) \\
& = & [m]^\ast \Big(z Q\big(-\vartheta -\tfrac 1m\big)-R(-\vartheta - \tfrac 1m)\Big),
\end{eqnarray*}
where we have used the elementary identity $z^m\partial^m = \vartheta(\vartheta - 1)\cdots(\vartheta-m+1)$ to pass from the second equality to the third. This calculation shows that the operators $\FT(\partial^mL)$ and $[m]^\ast K$ are equal up to a sign, and hence that there is an isomorphism of $\sD_{\IA^1}$-modules as claimed.
\end{par}
\begin{par}
Pochhammer's theorem \cite[Prop.\,2.8]{BH} states that the dimension of the space of vanishing cycles of $A_0$ at its non-zero singularity is one-dimensional. Since the map $z\to z^m$ is \'etale away from zero, the vanishing cycles of $\FT(M_1)$ at any point $z \neq 0$ are isomorphic to the vanishing cycles of $A_0$ at the point $z^m$, hence the equality
\begin{equation}
\sing(\FT(M_1)) = \sing([m]^\ast A_0) = c[0] + \sum_{k=0}^{m-1}[m \cdot \rho \cdot e^{2\pi i k/m}]    
\end{equation}
for some integer $c$. Besides, the kernel and the cokernel of the adjunction map $M_1 \to  j_\ast j^\ast M_1$ are Dirac modules supported at $0$. After Fourier transform, this yields an exact sequence
$$0 \to \cO_{\IA^1}^a \to \FT(M_1) \to \FT(j_\ast j^\ast M_1)\to \cO_{\IA^1}^b \to 0 $$
for some integers $a$ and $b$. Theorem \ref{Thm:PropertiesOfGModules} shows that $\FT(j_\ast j^\ast M_1)$ is of type $G$, so that~$j^\ast M_1$ is of type $E$. Moreover, the equality 
\[
\sing(\FT(M_1)) = \sing(\FT(j_\ast j^\ast M_1)) = \div(j^\ast M_1)
\] follows from this exact sequence. Next, there is a short exact sequence 
\[
0 \to \cO_{\IG_m}^m \to j^\ast M_1 \to M \to 0
\] of $\sD_{\IG_m}$\nobreakdash-modules, which shows that $M$ is of type $E$ as claimed, and that its divisor is equal to~$\div(M) = \div(j^\ast M_1)-m[0]$. Finally, the formula for $\div(M)$ follows from the equality~\hbox{$p=c-m$}, which can be deduced from \hbox{$\deg(\sing(\FT(j_\ast M))) = \dim M = q$}. 
\end{par}
\end{proof}

\begin{defn}
We denote by $\bH$ the tannakian subcategory of $\bE$ generated by all $E$\nobreakdash-hypergeometric $\sD_{\IG_m}$-modules as in Definition \ref{Defn:E-hypergeometricmodule}, and by $\cH$ the subalgebra of $\cE$ generated by all solutions of $E$\nobreakdash-hypergeometric operators.
\end{defn}

\begin{para}%
The category $\bH$ contains the Kummer modules $K(a)$ for $a\in \IQ$, as well as the module~$\sD_{\IG_m}/\sD_{\IG_m}\vartheta^2$ annihilating the logarithm function, and hence the algebra $\cH$ contains~$z^a$ and $\log(z)$. Indeed, these functions are solutions of the hypergeometric operators
$$\Hyp(t-a, t^2-at) = (\vartheta-z)(\vartheta-a) \qqet \Hyp(t^2,t^3) = (\vartheta-z)\vartheta^2$$
respectively. The fact \eqref{Eqn:AdjHyp} that the adjoint of a hypergeometric operator is again hypergeometric implies that the dual of an $E$-hypergeometric module is again an $E$\nobreakdash-hypergeometric module. Each object of $\bH$ can thus be written as a subquotient of some tensor construction on the given generators without involving duals. The subcategory $\bH\subseteq \bE$ and the subalgebra~$\cH \subseteq \cE$ determine each other, in the sense that an object $M$ of $\bE$ belongs to $\bH$ if and only if $M$ admits a basis of solutions in $\cH$, and conversely, $\cH$ can be retrieved from $\bH$ as 
\begin{equation}\label{Eqn:PresentationOfAlgebraH}
\cH = \{s(m)\tq M\in \bH,\,m\in M,\, s\in \Hom_\sD(M,\cE)\}.  
\end{equation}
In particular, the algebra $\cH$ contains all polynomial expressions in hypergeometric $E$\nobreakdash-functions, and our ultimate goal is to produce an $E$\nobreakdash-function that does not belong to $\cH$. Lemma \ref{Lem:NoOrAllSolutionsInH} states that an irreducible object $M\in \bE$ belongs to $\bH$ as soon as just one of the entries of a fundamental matrix of solutions of $M$ belongs to $\cH$. 
\end{para}

\begin{prop}\label{Pro:SemisimpleObjectsOfH}
The tannakian subcategory of $\bH$ generated by Kummer modules $K(a)$ and by $E$-hypergeometric modules $H^E(P,Q) = [m]^\ast H(P,Q)$ such that the hypergeometric module $H(P,Q)$ is simple and not Kummer induced contains every semisimple object of $\bH$.
\end{prop}

\begin{proof}
Let $M\in \bH$ be a simple object. By definition of $\bH$, the object $M$ is a subquotient of a tensor construction made from $E$-hypergeometric modules. We need to establish $A$ as a subquotient of a tensor construction made from Kummer modules and $E$-hypergeometric modules of the given special type. To do so, we can without loss of generality suppose that $M$ is a subquotient of some $E$-hypergeometric module $H^E(P,Q)=[m]^\ast H(P,Q)$ of type $(p,q)$. By \cite[Cor.\,3.7.5.2]{Katz}, the semisimplification of $H(P,Q)$ is given by
$$H(P,Q)^{\mathrm{ss}} \simeq H(P_0,Q_0) \oplus \bigoplus_{i=1}^r K(a_i)$$
for some $r\geq 0$ and $a_i\in \IQ$, where $H(P_0,Q_0)$ is simple of type $(p-r,q-r)$. Since $M$ is simple, it is either one of the Kummer modules $[m]^\ast K(a_i) = K(ma_i)$, in which case we are done, or a subquotient of $H^E(P_0,Q_0) = [m]^\ast H(P_0,Q_0)$. We can thus assume from now on that $H(P,Q)$ is simple. If $H(P,Q)$ is Kummer induced, let $d\geq 2$ be the largest integer such that $H(P,Q)$ is in the image of $[d]_\ast$. By Katz's Kummer induction formula~\hbox{\cite[(3.5.6.1)]{Katz}}, the module $H(P,Q)$ is then of type $(p_1d,q_1d)$, and there is an isomorphism
$$H(P,Q) \cong [d]_\ast H(P_1,Q_1),$$
where $H(P_1,Q_1)$ is a hypergeometric $\sD$-module that is not Kummer induced. Since $H(P,Q)$ is simple, $H(P_1,Q_1)$ is simple too. Set $m=m_1d$ with $m_1=q_1-p_1$. There is an isomorphism
$$H^E(P,Q) = [m_1]^\ast[d]^\ast [d]_\ast H(P_1,Q_1) = \bigoplus_{\zeta^d=1}[m_1]^\ast H(\zeta P_1,Q_1) = \bigoplus_{\zeta^d=1} H^E(\zeta P_1,Q_1),$$
on noting the identity $[\zeta z]^\ast H(P_1,Q_1) = H(\zeta P_1,Q_1)$. Since $z\mapsto \zeta z$ is invertible, the hypergeometric modules $H(\zeta P_1,Q_1)$ are all simple, none of them is Kummer induced, and~$A$ is a subquotient of one of the associated $E$-hypergeometric modules.
\end{proof}

\vspace{14mm}
\section{\texorpdfstring{Galois theory of hypergeometric $\sD$-modules}{Galois theory of hypergeometric D-modules}}

\begin{par}
The Galois theory of hypergeometric differential equations is largely understood thanks to the work of Katz \cite{Katz}. In this section, we explain how this can be used to find constraints on the singularities of the Fourier transforms of certain $\sD$-modules in the category $\bH$. 
\end{par}

\begin{thm}[Katz]\label{Thm:KatzClassification}
Let $0 \leq p <q$ be integers, and set $m=q-p$. Consider an $E$\nobreakdash-hypergeometric $\sD$\nobreakdash-module $H^E(P,Q)=[m]^\ast H(P,Q)$ of type $(p, q)$ such that $H(P,Q)$ is simple and not Kummer-induced. Let $G\subseteq \GL_q$ denote its differential Galois group. Then $G$ is reductive, and its determinant $\det(G)\subseteq \IG_m$ is finite if and only if $q>p+1$. The derived group of the connected component $G^\circ$ of $G$ satisfies
$$G^\circ = \begin{cases}
G^{\circ, \mathrm{der}} & \mbox{ if $\det(G)$ is finite}, \\
\IG_m \cdot G^{\circ, \mathrm{der}} & \mbox{ if $\det(G) = \IG_m$.}
\end{cases}$$
Up to conjugation, the possibilities for $G^{\circ, \mathrm{der}} \subseteq \SL_q$ are the following: 
\begin{enumerate}
    \item If $p-q$ is odd, then $G^{\circ, \mathrm{der}} = \SL_q$.
    \item If $p-q$ is even, then $G^{\circ, \mathrm{der}}$ is $\SL_q$, $\SO_q$ or, if $q$ is even, $\Sp_q$.
    \item If $(p,q) = (1,7)$, then $G^{\circ, \mathrm{der}}$ can as well be the exceptional group $G_2$ inside $\SL_7$ via its standard representation of dimension $7$.
    \item If $(p,q) = (2,8)$, then $G^{\circ, \mathrm{der}}$ can as well be $\mathrm{Spin}_7$ inside $\SL_8$ via its standard representation of dimension $8$, or $\PSL_3$ as the image of $\SL_3$ in its adjoint representation, or~$(\SL_2)^3$ in the tensor product of the standard representations of each factor.
    \item If $(p,q) = (3,9)$, then $G^{\circ, \mathrm{der}}$ can as well be the image of $\SL_3\times \SL_3$ in the tensor product of the standard representations of each factor.
\end{enumerate}
\end{thm}

\begin{proof}
The Galois group of $H^E(P,Q)$ is a subgroup of finite index of the differential Galois group of $H(P,Q)$, as is the case for any pullback of a vector bundle with connection by an étale cover according to \cite[Prop.\,1.4.5]{KatzGalois}. It thus suffices to prove the theorem when~$G$ is the differential Galois group of the simple and non-Kummer induced hypergeometric $\sD$\nobreakdash-module~$H(P,Q)$ instead. The result is then the combination of \cite[Th.\,3.6]{Katz}, which in addition to all the above statements lists two more possibilities for $G^{\circ, \mathrm{der}}$ that could a priori occur for $(p, q)=(2, 8)$, and~\cite[Prop.\,4.0.1]{Katz}, which says that those do not actually occur. 
\end{proof}

\begin{para}
To get a hold on the Galois groups of more general objects of $\bH$, we use a version of Goursat's lemma that is valid in any tannakian category. Let $\bT$ be a neutral tannakian category with unit object $\II$ over some field of characteristic zero, equipped with a fixed fibre functor, and let $A$ and $B$ be objects of $\bT$ with Galois groups $G_A$ and $G_B$. The Galois group~$G_{A \oplus B}$ of~\hbox{$A\oplus B$} is then a subgroup of the product $G_A \times G_B$ with the property that the projection maps \hbox{$p_A\colon G_{A \oplus B}\to G_A$} and $p_B\colon G_{A \oplus B}\to G_B$ are surjective, and similarly for the Lie algebras $\fg_{A \oplus B} \subseteq \fg_A\times \fg_B$. The object $A$ is contained in the tannakian category generated by~$B$ if and only if the map $p_B\colon G_{A \oplus B}\to G_B$ is an isomorphism, in which case we say that~\emph{$A$ is generated by $B$}. Since we are in characteristic zero, $A\otimes A^\vee = \underline{\End}(A)$  canonically splits as
$$\underline{\End}(A) = \underline{\End}^0(A) \oplus \II,$$
where $\underline{\End}^0(A) \subseteq \underline{\End}(A)$ is the kernel of the evaluation map $A\otimes A^\vee \to \II$, and $\II$ sits inside the endomorphisms as the image of the counit map $\II \to A\otimes A^\vee$.
\end{para}

\begin{defn}\label{Def:Lie-generated}
Let $A$ and $B$ be objects of a tannakian category $\bT$. Write $\fg_B$ and~$\fg_{A\oplus B}$ for the Lie algebra of the Galois group of $B$ and $A\oplus B$ respectively. We say that $A$ is \emph{Lie-generated} by $B$ if the projection map $\fg_{A\oplus B}\to\fg_B$ is an isomorphism. 
\end{defn}

\begin{para}
Here are two illustrations of this definition. First, consider an arbitrary object $B$ and an object $C$ with finite Galois group. The Galois group of $B \oplus (B\otimes C)$ is a finite cover of both the Galois group of $B$ and the Galois group of $B\otimes C$. It follows that an object $A$ is Lie-generated by $B$ if and only if it is Lie-generated by $B\otimes C$. This applies in particular to the case where $\bT$ is the category of connections on $\IG_m$ and $C$ a Kummer module. Second, consider objects $B$ and $C$, and suppose that $\det(C)$ has a finite Galois group. Then, the Galois group of $B\oplus C$ is a finite cover of the Galois group of $B\otimes C$, and hence $A$ is Lie-generated by $B\oplus C$ if and only if $A$ is Lie-generated by $B\otimes C$.
\end{para}

\begin{lem}\label{Lem:LieGeneration}
Let $\bT$ be a tannakian category generated by a family of objects $\cB$. Let $A$ be an object of $\bT$ such that the Lie algebra of the Galois group of $A$ is simple and non\nobreakdash-commutative. Then $A$ is Lie-generated by a single object of the family $\cB$.
\end{lem}

\begin{proof}
\begin{par}
The object $A$ is isomorphic to a subquotient of some tensor construction of objects of the family $\cB$. As any such tensor construction involves only finitely many objects, we can suppose without loss of generality that $\cB$ consists of finitely many objects. In particular, $A$ is Lie-generated by a finite sum of objects of $\cB$, say by $B = B_1\oplus\cdots\oplus B_n$, where~$n \geq 1$ is chosen minimal for this property. If $n=1$, we are done. If $n\geq 2$, it is enough to show that $A$ is Lie-generated by $B_1$ or by $B_2\oplus\cdots\oplus B_n$, so we can assume $n=2$.
\end{par}

\begin{par}
Writing $\fg_1, \fg_2, \fg_A, \fg_B$ for the Lie algebras of the Galois groups of $B_1, B_2, A,$ and $B = B_1\oplus B_2$ respectively, there are canonical morphisms of Lie algebras 
$$\begin{diagram}
\setlength{\dgARROWLENGTH}{6mm}
\node{\fg_B}\arrow{s,l}{p_A}\arrow[2]{e,t}{\:\:\subseteq\:\:}\node[2]{\!\! \fg_1 \times \fg_2\!\!\!\!} \arrow{s,r}{\mathrm{proj}_i}\\
\node{\fg_A}\node[2]{\fg_i.}
\end{diagram}$$
Let $\mathfrak k_i \trianglelefteq \fg_B$ denote the kernel of the canonical projection $p_i\colon \fg_B\to\fg_i$. Since the image of~$\mathfrak k_i$ in~$\fg_A$ is an ideal, it is either zero or the whole $\fg_A$. In case one of the $\mathfrak k_i$ maps to zero in $\fg_A$, the projection~$p_A$ factors over $\fg_i$, and hence $A$ is Lie-generated by $B_i$. Suppose then that both ideals $\mathfrak k_1$ and $\mathfrak k_2$ surject onto $\fg_A$. In the product $\fg_1\times \fg_2$, the ideals $\{0\}\times \fg_2$ and~$\fg_1\times \{0\}$ commute. Thus, $\mathfrak k_1$ and $\mathfrak k_2$ commute in $\fg_B$, and hence $\fg_A$ is commutative.
\end{par}
\end{proof}

\begin{lem}\label{Lem:End0Isomorphism}
Let $A$ and $B$ be three-dimensional objects of a tannakian category. Suppose that the Galois group of $A$ is isomorphic to $\SL_3$ and that $A$ is Lie-generated by $B$. Then, there exists an isomorphism $\underline{\End}^0(A)  \cong \underline{\End}^0(B)$.
\end{lem}

\begin{proof}
Let $G$ be the Galois group of $A \oplus B$. For our purposes, we can assume that the ambient tannakian category is the category of representations of $G$. The objects $A$ and $B$ correspond to three-dimensional representations
$$\rho_A\colon G \to \GL(V_A) \simeq \GL_3 \qqet \rho_B\colon G\to \GL(V_B) \simeq \GL_3,$$
whose images $G_A = \rho_A(G)$ and $G_B = \rho_B(G)$ are the Galois groups of $A$ and $B$. By Goursat's Lemma \cite[p.\,75]{Lang}, there exist normal subgroups $N_A \subseteq G_A$ and $N_B\subseteq G_B$ and an isomorphism
$$\alpha\colon G_A/N_A \stackrel{\sim}{\longrightarrow} G_B/N_B$$
such that the image of the injective morphism $G\to G_A \times G_B$ consists of those pairs $(g_A,g_B)$ satisfying $\alpha(g_AN_A) = g_BN_B$. Since $A$ is Lie-generated by $B$, the subgroup $N_A$ is finite, and since $G_A$ is isomorphic to $\SL_3$, the group $G_A/N_A$ is isomorphic to either $\SL_3$ or $\PSL_3$. The group $G_B \subseteq \GL_3$ is therefore isomorphic to $\GL_3$ or to $\mu_n\SL_3$ for some integer $n\geq 1$. Either way, the image $\overline G_A$ of $G_A$ in the representation $\underline{\End}^0(V_A)$ is isomorphic to $\PSL_3$, and so is the image $\overline G_B$ of $G_B$ in the representation $\underline{\End}^0(V_B)$. The image $\overline G$ of $G$ in the representation 
$$\underline{\End}^0(V_A) \oplus \underline{\End}^0(V_B)$$
is the graph of the isomorphism $\overline \alpha\colon \overline G_A \to \overline G_B$ induced by $\alpha$. In particular, $\overline G$ is isomorphic to $\PSL_3$ as well. The representations $\underline{\End}^0(V_A)$ and $\underline{\End}^0(V_B)$ of $\overline G$ are both irreducible and of dimension $8$. Since $\PSL_3$ has up to isomorphism only one such representation, we conclude that there exists an isomorphism of $\overline G$-representations, and therefore of~$G$\nobreakdash-representations~\hbox{$\underline{\End}^0(V_A) \cong \underline{\End}^0(V_B)$} as claimed.
\end{proof}

\begin{thm}\label{Thm:LieGenerationByHypergeometric}
Let $M$ be an object of $\bE$ whose Galois group is isomorphic to $\SL_3$. If~$M$ is Lie-generated by objects of $\bH$, then~$M$ is Lie\nobreakdash-generated by a single $E$-hypergeometric $\sD$-module of type $(p,3)$ with Galois group equal to $\SL_3$ if $p \in \{0,1\}$, and to $\GL_3$ if $p=2$.
\end{thm}

\begin{proof}
\begin{par}
The object $M$ is semisimple. By Proposition \ref{Pro:SemisimpleObjectsOfH} and Lemma \ref{Lem:LieGeneration} applied to the category of all semisimple objects of $\bH$, the module $M$ is Lie-generated by some $E$-hypergeometric module $N=[m]^\ast H(P,Q)$ such that $H(P,Q)$ is simple and not Kummer induced. The Lie algebra of the Galois group of $N$ surjects onto the Lie algebra $\sl_3$. By Theorem \ref{Thm:KatzClassification}, $N$ is thus an~$E$\nobreakdash-hypergeometric module of type $(p,q)$ with Galois group $G\subseteq \GL_q$ of the following kind:
\begin{enumerate}
    \item $q = 3$ and $G^\circ = \SL_3$ if $p \in \{0,1\}$, or $G=\GL_3$ if $p=2$.
    \item $(p,q)=(2,8)$ and $G^\circ = \PSL_3$ as the image of $\SL_3$ in its adjoint representation.
    \item $(p,q)=(3,9)$ and $G^\circ$ is the image of $\SL_3\times \SL_3$ in the tensor product of the standard representations of each factor.
\end{enumerate}
In the first case, if $p \in \{0,1\}$ we can twist $N$ by a Kummer module to ensure that $G$ is equal to~$\SL_3$. It remains to explain why the exceptional cases of type $(2,8)$ or $(3,9)$ are not needed. 
\end{par}

\begin{par}
{\bf Case $(2,8)$.} Suppose that $N = [6]^\ast H(P,Q)$ is of type $(2,8)$ and that the connected component of the identity in its Galois group is $\PSL_3$. After replacing $H(P,Q)$ with an appropriate Kummer twist, we may assume $G$ is contained in $\mathrm{O}_8$. By \cite[Lem.\,4.3.2]{Katz}, there exists then a hypergeometric module $H(P_0,Q_0)$ of type $(0,3)$ with Galois group $\SL_3$ and an isomorphism
$$[2]^\ast H(P,Q) \cong \underline{\End}^0(H(P_0,Q_0)).$$ 
Pulling back by $z\mapsto z^3$, it follows that $N$ belongs to the tannakian category generated by the $E$\nobreakdash-hypergeometric module $[3]^\ast H(P_0,Q_0)$, and hence $M$ is Lie-generated by $H^E(P_0,Q_0)$. 
\end{par}

\begin{par}
{\bf Case $(3,9)$.} Suppose that $N = [6]^\ast H(P,Q)$ is of type $(3,9)$ and that the connected component of the identity in its Galois group is the image of $\SL_3\times \SL_3$ in $\GL_9$. By \cite[Lem.\,4.6.1 and~4.6.3]{Katz}, after replacing $H(P,Q)$ with a suitable Kummer twist, there exist hypergeometric modules $H(P_1,Q_1)$ and $H(P_2,Q_2)$ of type $(0, 3)$ and Galois group $\SL_3$ and an isomorphism 
$$[2]^\ast H(P,Q) \cong H(P_1,Q_1) \otimes H(P_2,Q_2),$$
which yields $N \cong H^E(P_1,Q_1) \otimes H^E(P_2,Q_2)$ after pullback by $z\mapsto z^3$. Since $M$ cannot be Lie-generated by a Kummer module, it is Lie-generated by $H^E(P_1,Q_1)$ or by $H^E(P_2,Q_2)$ thanks to Lemma~\ref{Lem:LieGeneration}. This concludes the proof. 
\end{par}
\end{proof}

\begin{lem}\label{Lem:Factiod}
Let $s_1,s_2,s_3$ be complex numbers whose differences $s_i-s_j$, for~\hbox{$i\neq j$}, are the vertices of a regular hexagon. Then $s_1,s_2,s_3$ are the vertices of an equilateral triangle.
\end{lem}

\begin{proof}
Since the elements of $S = \{s_i-s_j\tq i\neq j\}$ sum to zero, if they are the vertices of a regular hexagon, then this hexagon is centered at $0 \in \IC$, and each vertex is equidistant from~$0$. Then $\abs{s_1-s_2} = \abs{s_1-s_3} = \abs{s_2-s_3}$
holds, so~$s_1,s_2,s_3$ indeed form an equilateral triangle. 
\end{proof}

\begin{thm}\label{Thm:symconst}
Let $M$ be a three-dimensional object of $\bE$ whose Galois group contains~$\SL_3$. If $M$ is Lie-generated by objects of $\bH$, then the singularities of $\FT(j_\ast M)$ are either collinear or form an equilateral triangle.
\end{thm}

\begin{proof}
\begin{par}
We assume, as we may, that the Galois group of $M$ is equal to $\SL_3$. Indeed, $\det(M)$ takes the form $E(s) \otimes K(a)$ for some $s \in \IC$ and $a \in \IQ$ by the classification of rank-one objects of~$\bE$. Twisting $M$ by $E(-s/3) \otimes K(-a/3)$ yields an object with trivial determinant and only changes by translation the singularities of the Fourier transform. Set $\div(M)=[s_1] + [s_2] + [s_3]$. Since~$\det(M)$ is trivial, the relation $s_1+s_2+s_3=0$ holds. By Theorem \ref{Thm:LieGenerationByHypergeometric}, the object $M$ is Lie-generated by an $E$-hypergeometric $\sD$-module $N=H^E(P,Q)$ of type $(p, 3)$ with Galois group~$\SL_3$ if $p \in \{0, 1\}$, and $\GL_3$ if $p=2$. By Lemma \ref{Lem:End0Isomorphism}, there is an isomorphism
$$\underline{\End}^0(M) \cong \underline{\End}^0(N).$$
Writing $\div(N)=[t_1]+[t_2] + [t_3]$, the existence of this isomorphism and the formulas for divisors given in \ref{Para:FibreFunctor} show that the equality
\begin{equation}\label{Eqn:SingDivisorOfEquality}
2[0] + \sum_{i\neq j}[s_i-s_j] = 2[0] + \sum_{i\neq j}[t_i-t_j]
\end{equation}
holds in $\IZ[\overline\IQ]$. In the remainder of the proof, we discuss the three values of $p$ separately.
\end{par}

\begin{par}{\bf Case $p=0$.} In that case, $\div(N) = [\lambda] + [\lambda \zeta^2] + [\lambda \zeta^4]$ holds for some $\lambda \in \overline\IQ^\times$ and $\zeta=e^{2\pi i/6}$ by Theorem \ref{Thm:EHypergeometricOperators}. We deduce the equality
$$\sum_{i\neq j}[s_i-s_j] = \sum_{k=0}^5[\sqrt {-3} \cdot \lambda \cdot \zeta^k]$$
from \eqref{Eqn:SingDivisorOfEquality}. Since the right-hand side divisor is supported on a regular hexagon, $s_1,s_2,s_3$ lie on an equilateral triangle by Lemma \ref{Lem:Factiod}. 
\end{par}

\begin{par}{\bf Case $p=1$.} In that case, $\div(N)= [-\lambda] + [0] + [\lambda]$ holds for some $\lambda \in \overline\IQ^\times$, again by Theorem~\ref{Thm:EHypergeometricOperators}. The 
equality \eqref{Eqn:SingDivisorOfEquality} then reads
$$\sum_{i\neq j}[s_i-s_j] = [-2\lambda] + 2[-\lambda] + 2[\lambda] + [2\lambda],$$
which is only possible if $[s_1]+[s_2]+[s_3] = [-\lambda] + [0] + [\lambda]$. Thus, $s_1,s_2,s_3$ lie on the line $\IR \lambda$.
\end{par}

\begin{par}{\bf Case $p=2$.} In that case, $N$ is a hypergeometric module with divisor $\div(N)= 2[0]+[\lambda]$ for some $\lambda \in \overline\IQ^\times$. The 
equality \eqref{Eqn:SingDivisorOfEquality} reads
$$\sum_{i\neq j}[s_i-s_j] =  2[-\lambda] + 2[0] + 2[\lambda],$$
which is only possible if $[s_1]+[s_2]+[s_3] = 2[-\mu] + [2\mu]$ for $\mu = \lambda/3$ or $\mu = -\lambda/3$. In either case, $s_1,s_2,s_3$ lie on the line $\IR \lambda$.
\end{par}
\end{proof}

\vspace{14mm}
\section{\texorpdfstring{A non-hypergeometric $E$-function}{A non-hypergeometric E-function}}

\begin{par}
Armed with the symmetry constraint from Theorem \ref{Thm:symconst}, it remains to write down a concrete example of an $E$\nobreakdash-function that is not a polynomial expression in hypergeometric $E$-functions. For this, we shall consider a family of $\sD_{\IG_m}$-modules attached to polynomials of degree four. 
\end{par}

\vspace{4mm}
\begin{prop}\label{Pro:DModuleForPolynomial}
Let $n \geq 1$ be an integer, and set $\zeta=\exp(2\pi i/ (n+1))$. Let $f \in \overline\IQ[t]$~be a monic polynomial of degree $n+1$, which we view as a morphism $f\colon\IA^1 \to \IA^1$. Consider the $\sD$\nobreakdash-modules $A$ on $\IA^1$ and $M$ on $\IG_m$ defined as
$$A = f_\ast\cO_{\IA^1}/\cO_{\IA^1} \qqet  M=j^\ast\FT^{-1}(A).$$
The $\sD_{\IG_m}$-module $M$ is of type $E$. With respect to appropriate bases, the $n$-by-$n$ matrix of functions $F=(F_{a, b})_{1\leq a,b\leq n}$ defined for $ Re(z)>0$ by the convergent integrals
\[
F_{a, b}(z)= \int_0^\infty \exp(-zf(x))x^{b-1}dx - \zeta^{ab}\int_0^\infty \exp(-zf(\zeta^ax))x^{b-1}dx
\]
is a fundamental matrix of solutions of $M$. In particular, all the functions $F_{a, b}$ belong to $\cE$. \end{prop}

\begin{proof} It was already proved in Example \ref{Exmp:ExponentialPeriodModules} that $M$ is a $\sD_{\IG_m}$-module of type $E$ and that the identity $A=\FT(j_\ast M)$ holds. By the equivalence between conditions (1) and (2) in Proposition \ref{Pro:WishListC}, the module $M$ admits a basis of solutions in $\cE$, and hence a fundamental matrix of solutions with respect to any choice of bases has entries in $\cE$. 

\begin{par}Let $S$ be the set of critical values of $f$. Elements of the $\sD$-module 
\[
f_\ast\cO_{\IA^1} = \overline\IQ[z,u]/(f(u)-z)
\] can be thought of as algebraic functions in the complex variable $z$, defined on a small open disk around a point in $\IC\setminus S$. Concretely, fix $z_0 \in \IC\setminus S$ and choose a root $u_0$ of the polynomial $f(x)-z_0$. By the implicit function theorem, there exists a unique holomorphic function~$u(z)$ defined in an open neighbourhood of $z_0$ satisfying $u(z_0)=u_0$ and $f(u(z))-z=0$ for all~$z\in U$. This function $u$ and its powers are explicitly given by
\[
u^p(z)=\frac{1}{2\pi i}\int_\gamma \frac{x^pf'(x)}{f(x)-z}dx, 
\]
where the contour $\gamma$ is a simple loop around $u_0$ not enclosing any roots of $f(x)-z_0$ other than~$u_0$. For a polynomial $g \in \overline\IQ[x]$, define
\[
G_\gamma(g,z)=\frac{1}{2\pi i}\int_\gamma \frac{g(x)}{f(x)-z}dx
\]
and observe the relations
$$G_\gamma(g',z) = \frac{1}{2\pi i}\int_\gamma \frac{g'(x)}{f(x)-z}dx = \frac{1}{2\pi i}\int_\gamma \frac{g(x)f'(x)}{(f(x)-z)^2}dx = \frac\partial{\partial z} G_\gamma(gf',z)$$
obtained by integration by parts and by differentiating under the integral. We now have two~$\IC[z]$\nobreakdash-bases of $f_\ast\cO_{\IA^1}$ at hand, the first one given by the functions $u^p= G_\gamma(x^pf',z)$, and the second one by the functions $G_\gamma(x^p,z)$ for $p=0,1,\ldots, n-1$. In particular, $G_\gamma(f',z)$ is the constant function with value $1$, which generates the image of the adjunction map $\cO_{\IA^1} \to f_\ast\cO_{\IA^1}$. The~$\sD$\nobreakdash-module~$A$ has therefore the following presentation: it is generated by symbols~$G(g,z)$, one for each\hbox{~$g\in \overline\IQ[x]$}, subject to the following relations:
\begin{enumerate}
\item Linearity: $G(ag+bh,z) = aG(g,z) +bG(h,z)$ for $a,b\in \overline\IQ$ and $g,h\in \overline\IQ[x]$.
\item Multiplication rule: $zG(g,z) = G(fg,z)$ for all $g\in \overline\IQ[x]$.
\item Derivation rule: $\partial G(gf',z)= G(g',z)$ for all $g\in \overline\IQ[x]$.
\item $G(f',z)=0$.
\end{enumerate}
The inverse Fourier transform $\FT^{-1}(A)$ has the following dual presentation, obtained by replacing~$\partial$ with~$z$, and $z$ with $-\partial$ in the above: it is the $\sD$-module generated by symbols~$E(g,z)$, one for each~$g\in \overline\IQ[x]$, modulo the relations:
\begin{enumerate}
\item Linearity: $E(ag+bh,z) = aE(g,z) +bE(h,z)$ for $a,b\in \overline\IQ$ and $g,h\in \overline\IQ[x]$.
\item Derivation rule: $\partial E(g,z) = -E(fg,z)$ for all $g\in \overline\IQ[x]$.
\item Multiplication rule: $z E(gf',z) = E(g',z)$ for all $g\in \overline\IQ[x]$.
\item $E(f',z)=0$.
\end{enumerate}
Let $\cA$ denote the differential algebra of holomorphic functions on the half-plane $\Re(z)>0$. To give a solution $\FT(A)\to \cA$ amounts to specifying, for each generator $E(g,z)$, an element of $\cA$ compatible with relations. To do so, let $\gamma_a\colon \IR\to\IC$ be the path defined by $\gamma_a(t) = -t\zeta^a$ for $t\leq 0$ and $\gamma_a(t)=t$ for $t\geq 0$, and set
\[
E_{\gamma_a}(g,z) = \int_{\gamma_a} e^{-zf(x)}g(x)dx
\]
for polynomials $g\in \overline\IQ[x]$. For $g=x^{b-1}$, this yields the function $F_{a, b}$ in the statement of the proposition. The integral converges, and there is indeed no problem in exchanging differentiation with respect to $z$ and integration, due to the exponential decay of $e^{-zf(\zeta^ax)}$ as~$x\to \infty$, which follows from the assumption that $f$ is monic and $\Re(z)>0$. The expression~$E_{\gamma_a}(g,z)$ is linear in $g$, and the derivation rule is obtained by differentiating under the integral sign. The multiplication rule follows from integration by parts
\[
\int_{\gamma_a} e^{-zf(x)}g'(x)dx = z \int_{\gamma_a} e^{-zf(x)}f'(x)g(x)dx, 
\]
on noting that $e^{-zf(x)}g(x)$ tends to $0$ as $x\to 1\infty$ or $x\to \zeta^a\infty$. Setting $g=1$, we find~$zE_{\gamma_a}(f',z)=0$, and hence
$E_{\gamma_a}(f',z)=0$ as needed. Each column of the matrix $F(z)$ is indeed a solution vector for $M$ with respect to the basis $E(x^p,z)$ for $p=0,1,\ldots,n-1$. That~$F(z)$ is a fundamental matrix of solutions means that its determinant does not vanish. This follows from the fact that, for fixed $z$ with positive real part, the paths $\gamma_1, \dots, \gamma_n$ form a basis of the rapid decay homology $\rH_1^{\mathrm{rd}}(\IA^1, zf)$, the differential forms $dx, xdx, \dots, x^{n-1}dx$ form a basis of the twisted de Rham cohomology $\rH^1_{\dR}(\IA^1, zf)$, and the integration $(\gamma, \omega) \mapsto \int_\gamma e^{-zf} \omega$ induces a perfect pairing between these spaces by a theorem of Bloch--Esnault \cite[Th.\,0.1]{BlochEsnault2}. This special case actually goes back to a 1976 letter from Deligne to Malgrange \cite[p.\,17]{DMR}, and one can even show the equality $\det F(z) = \beta e^{-z\alpha}$, where $\beta$ is some non-zero complex number and~$\alpha$ is the sum of the critical values of~$f$, see \cite[Prop.\,5.4]{BlochEsnault}.
\end{par}
\end{proof}

\begin{lem}\label{Lem:simplicity}
Let $f \in \overline\IQ[t]$~be a polynomial of degree $n+1$. If $f\colon\IA^1\to\IA^1$ has $n$ distinct critical values, then the $\sD_{\IA^1}$-module $A = f_\ast\cO_{\IA^1}/\cO_{\IA^1}$ is simple.
\end{lem}

\begin{proof}
Let $S$ be the set of critical values of $f$, set $U = \IC\setminus S$, and choose a base point~\hbox{$x \in U$}. The monodromy of the \'etale covering of $U$ given by $f$ is generated by $n$ non-trivial permutations~$\tau_1, \tau_2, \ldots, \tau_n$ acting on the set of $n+1$ elements $f^{-1}(x)$, which can moreover be arranged in such a way that the product $\tau_\infty^{-1} = \tau_1\tau_2\cdots\tau_n$ is an $(n+1)$-cycle. Let $1\leq c_i \leq n$ be the number of cycles in the permutation $\tau_i$, and set $c_\infty=1$. Hurwitz's genus formula for $f$, now viewed as a ramified covering $f\colon\IP^1\to\IP^1$ of degree $n+1$, reads
$$2 = 2(n+1) - (n+1-c_\infty) - \sum_{i=1}^n(n+1-c_i),$$
which can only hold if $c_i=n$ for $i=1,2,\ldots,n$. In other words, $\tau_1, \ldots, \tau_n$ are all transpositions. Since a transitive subgroup of a symmetric group generated by transpositions is the whole group, the monodromy representation of the local system $f_\ast \underline \IQ$ on $U$ is the standard representation of the symmetric group on $n+1$ elements, which is well known to split into a trivial factor of dimension $1$ and a simple factor of dimension $n$. Via the Riemann--Hilbert correspondence, this decomposition yields~$f_\ast \cO_{\IA^1} = \cO_{\IA^1} \oplus A$, and hence $A$ is simple. 
\end{proof}

\begin{lem}\label{Lem:ContainsSL3}
Let $f \in \overline\IQ[t]$~be a polynomial of degree four. If the critical values of $f$ are not collinear, then the differential Galois group of the module $M=j^\ast\FT^{-1}( f_\ast\cO_{\IA^1}/\cO_{\IA^1})$ contains $\SL_3$.
\end{lem}

\begin{proof}
The polynomial $f$ must have three distinct critical values in order for them not to be collinear. Together with the equivalence of categories explained in \ref{Par:ConnAndRS}, Lemma \ref{Lem:simplicity} shows that the $\sD$\nobreakdash-module $M\in \Conn_0(\IG_m)$ is simple. Hence, its differential Galois group is a reductive subgroup of~$\GL_3$. By the discussion at the beginning of \ref{Para:FibreFunctor}, this group contains the torus with character group the subgroup of~$\overline\IQ$ generated by the critical values of $f$, which by assumption has rank at least two. Any reductive subgroup of $\GL_3$ of rank $\geq 2$ whose three-dimensional standard representation is simple contains $\SL_3$. 
\end{proof}

\begin{para}\label{sec:ComputationExponentialIntegral}
Let $f$ be a monic polynomial of degree four with algebraic coefficients, say given in the form $f(x) = x^4 - a_3x^3-a_2x^2-a_1x-a_0$. We wish to compute the integral
\[
P(z)=F_{2, 1}(z)=\int_{\IR} e^{-zf(x)}dx,
\]
which defines a holomorphic function on the half\nobreakdash-plane~\hbox{$\Re(z)>0$}. This function can be analytically continued to any simply connected domain in $\IC\setminus\{0\}$, but not across $0$, since it has finite monodromy of order four around this point. In other words, $P(z^4)$ is a meromorphic function with a single pole at $z=0$, as the following computation shows: 
\begin{align*}
P(z^4)&=\int_{\IR} e^{-z^4(x^4 - a_3 x^3- a_2x^2- a_1 x - a_0)}dx \\
&=\frac{1}{z}\int_{z\IR} e^{-s^4}e^{a_3z s^3} e^{a_2z^2s^2} e^{a_1z^3 s} e^{a_0z^4}ds\\
&=\frac{1}{z}\sum_{k_0,k_1,k_2,k_3 \geq 0} \!\!\frac{(a_3z)^{k_3}(a_2z^2)^{k_2}(a_1z^3)^{k_1}(a_0z^4)^{k_0}}{k_3!k_2!k_1!k_0!} \int_{z\IR} e^{-s^4}s^{3k_3+2k_2+k_1}ds\\
&=\sum_{\substack{k_0,k_1,k_2,k_3 \geq 0 \\ k_1+k_3 \text{ even}}} \frac{a_3^{k_3}a_2^{k_2}a_1^{k_1}a_0^{k_0}}{2 k_3!k_2!k_1!k_0!} \Gamma\big(\tfrac{3k_3 + 2k_2 + k_1 +1}{4}\big) z^{k_3+2k_2 + 3k_1 + 4k_0-1}. 
\end{align*}
We made the change of variables $s=zx$, expanded all exponential functions except $e^{-s^4}$, and evaluated the remaining integral in terms of the gamma function. Setting $k=k_0+k_1+k_2+k_3$, we can now reindex the resulting series as 
\begin{align*}
P(z^4) & = \sum_{n=0}^\infty\sum_{\ast=0}\frac{a_3^{k_3}a_2^{k_2}a_1^{k_1}a_0^{k_0}}{2k_3!k_2!k_1!k_0!} \Gamma\big(k-n+\tfrac{1}{4}\big) z^{4n-1}+ \sum_{n=0}^\infty\sum_{\ast=2}\frac{a_3^{k_3}a_2^{k_2}a_1^{k_1}a_0^{k_0}}{2k_3!k_2!k_1!k_0!} \Gamma\big(k-n-\tfrac{1}{4}\big) z^{4n+1}, 
\end{align*}
where the sums labelled with $(\ast=r)$ run over the finitely many integers \hbox{$k_0, k_1, k_2, k_3 \geq 0$} satisfying~$k_3+2k_2+3k_1 + 4k_0-4n= r$. This yields the expression 
\begin{equation}\label{eqn:monodromydecomposition}
P(z)=\frac{1}{2}\Gamma\big(\tfrac{1}{4}\big)z^{-1/4}E_0(z)+ \frac{1}{2}\Gamma\big(\tfrac{-1}{4}\big)z^{1/4}E_2(z)
\end{equation}
for the original function, where $E_r(z)$ is the $E$-function
\begin{equation}\label{Eqn:EfunctionEr}
E_r(z)=\sum_{n=0}^{\infty} \left(\sum_{\ast=r}\frac{a_3^{k_3}a_2^{k_2}a_1^{k_1}a_0^{k_0} \cdot (\tfrac{1-r}{4})_{k-n} }{k_3!k_2!k_1!k_0!}  \right) z^n.
\end{equation}
\end{para}

\begin{lem}\label{Lem:MonodromyDecomposition}
Let $M$ be a $\sD_{\IG_m}$-module and let $s\colon M \to \IC(\!(z)\!)(z^{1/n})$ be a solution for some integer $n \geq 1$. For each $m\in M$, the unique Laurent series $s_j(m) \in \IC(\!(z)\!)$ satisfying
$$s(m) = \sum_{j=0}^{n-1}z^{j/n}s_j(m)$$
give rise to solutions $z^{j/n}s_j\colon M\to \IC(\!(z)\!)(z^{1/n})$ of $M$ for $j=0, 1, \dots, n-1$. 
\end{lem}

\begin{proof}
Existence and uniqueness of the given decomposition of $s(m)$ follows from the fact that the powers $z^{j/n}$ form a basis $\IC(\!(z)\!)$-basis of $\IC(\!(z)\!)(z^{1/n})$, and each function $s_j\colon M\to \IC(\!(z)\!)$ is~$\overline\IQ[z,z^{-1}]$\nobreakdash-linear. For every $m\in M$, the equality
$$\sum_{j=0}^{n-1}z^{j/n}s_j(\partial m) = s(\partial m) = \partial s(m) = \sum_{j=0}^{n-1}\partial(z^{j/n}s_j(m)) = \sum_{j=0}^{n-1}z^{j/n}(\tfrac jnz^{-1}s_j(m)+ s_j(m)')$$
holds. From this and the $\IC(\!(z)\!)$-linear independence of the $z^{j/n}$, we deduce that the equality $z^{j/n}s_j(\partial m) = \partial(z^{j/n}s_j(m))$
holds for every $j$, so $z^{j/n}s_j$ is indeed a solution of $M$.
\end{proof}

\begin{lem}\label{Lem:NoOrAllSolutionsInH}
Let $M$ be an irreducible $\sD_{\IG_m}$-module of type $E$, let $m\in M$ be a non-zero element, and let $s\colon M\to \cE$ be a non-zero solution.
\begin{enumerate}
    \item If $s(m)$ belongs $\cH$, then $M$ belongs to $\bH$.
    \item If $s(m)$ is algebraic over $\cH$, then $M$  is Lie-generated by objects of $\bH$.
\end{enumerate}
\end{lem}

\begin{proof}
Suppose that $s(m)$ belongs to $\cH$. By \eqref{Eqn:PresentationOfAlgebraH}, there exist an object $M_0\in \bH$, an element~\hbox{$m_0\in M_0$}, and a solution $s_0\colon M_0\to \cH$ satisfying $s(m)=s_0(m_0)$. We suppose without loss of generality that the $\sD_{\IG_m}$-module $M_0$ is generated by $m_0$. Since $s$ is non-zero and $M$ is irreducible, $s$ is injective. Hence,~$m \in M$ and $s(m)\in \cH$ have the same annihilator ideal in~$\sD_{\IG_m}$, and this ideal contains the annihilator ideal of $m_0\in M_0$. There is thus a unique morphism of $\sD_{\IG_m}$-modules $M_0 \to M$ mapping $m_0$ to $m$. Since $m$ is non-zero and $M$ is irreducible, this morphism is surjective. This shows that $M$ is a quotient of $M_0$, and hence belongs to~$\bH$ as claimed. The second statement is proved similarly, on noting that an object of~$\bE$ is Lie-generated by objects of $\bH$ if and only if all of its solutions are algebraic over $\bH$. 
\end{proof}

\begin{thm}\label{Cor:ApplicationtoPolynomials} 
Let $a_0,\ldots,a_3 \in \overline \IQ$ be algebraic numbers such that the critical values of the polynomial $f(x)= x^4-a_3x^3-a_2x^2-a_1x-a_0$ are neither collinear nor do they form an equilateral triangle. Then, the $E$-functions $E_0(z)$ and $E_2(z)$ given in \eqref{Eqn:EfunctionEr} are transcendental over $\cH$, and in particular are not polynomial expressions in hypergeometric $E$\nobreakdash-functions. 
\end{thm}

\begin{proof}
By Proposition~\ref{Pro:DModuleForPolynomial}, the module $M=j^\ast\FT^{-1}(f_\ast \cO_{\IA^1}/\cO_{\IA^1})$ is an object of $\bE$ and there exists an element $m_1\in M$, namely the class of $dx$, and a solution $s_2$ of $M$, namely integration along the real line, such that $F_{2,1} = s_2(m_1)$ holds. Let $r\in \{0,2\}$. By Lemma~\ref{Lem:MonodromyDecomposition}, there also exists a solution~\hbox{$s\colon M\to \cE$} with $s(m_1) = z^{(r-1)/4}E_r(z)$. Using Lemma~\ref{Lem:ContainsSL3} and Theorem~\ref{Thm:symconst}, the hypotheses on $f$ imply that $M$ is simple and not Lie-generated by objects of~$\bH$, so we deduce from Lemma \ref{Lem:NoOrAllSolutionsInH} that $z^{(r-1)/4}E_r(z)$, and hence $E_r(z)$, is transcendental over~$\cH$.
\end{proof}

\begin{para}
If we identify the space of monic polynomials of degree four with complex coefficients with $\IR^8$ by taking real and imaginary parts of coefficients, those with critical values that are neither collinear nor do they lie on an equilateral triangle form a Zariski open dense subset. In that sense, most monic polynomials $f\in \overline \IQ[t]$ of degree four satisfy this hypothesis. The $E$-function from the introduction was produced by choosing $f(x)=x^4-x^2+x$ in Theorem~\ref{Cor:ApplicationtoPolynomials}. The critical values of this polynomial are the roots of $x^3 +\tfrac 12 x^2 - \tfrac 12 x + \tfrac{23}{256}$, which form an isosceles, but not equilateral triangle. Since now $(a_3,a_2,a_1,a_0)$ is equal to $(0,1,-1,0)$, all terms in the sums labelled with $(\ast=r)$ in \eqref{Eqn:EfunctionEr} are zero, except when $k_3=k_0=0$, and we find the $E$-functions
\begin{displaymath}
E_r(z)=\sum_{n=0}^{\infty} \left(\sum_{2k_2+3k_1=4n+r}\frac{(-1)^{k_1} \cdot (\tfrac{1-r}{4})_{k-n} }{k_2!k_1!}  \right) z^n. 
\end{displaymath}
For $r=0$, the condition $2k_2+3k_1=4n+r$ implies that $k_1$ is even. Writing $2m=k_1$, it translates into $k_2=2n-3m$, hence the expression 
\[
E_0(z) =\sum_{n=0}^\infty \sum_{m=0}^{\floor{2n/3}} \frac{(\tfrac 14)_{n-m} }{(2n-3m)!(2m)!} \:z^n, 
\]
which is the one given in the introduction.
\end{para}

\bibliographystyle{amsplain}
    \bibliography{Siegel}

\providecommand{\eprint}[1]{\href{http://arxiv.org/abs/#1}{\texttt{arXiv\string:\allowbreak#1}}}\providecommand{\doi}[1]{\href{http://dx.doi.org/#1}{\texttt{doi\string:\allowbreak#1}}}
\providecommand{\bysame}{\leavevmode\hbox to3em{\hrulefill}\thinspace}
\providecommand{\MR}{\relax\ifhmode\unskip\space\fi MR }
\providecommand{\MRhref}[2]{%
  \href{http://www.ams.org/mathscinet-getitem?mr=#1}{#2}
}
\providecommand{\href}[2]{#2}
\begin{thebibliography}{10}

\bibitem{Gfunctions}
Y.~Andr\'{e}, \emph{{$G$}-functions and {G}eometry}, Aspects Math., Friedr.
  Vieweg \& Sohn, Braunschweig, 1989.

\bibitem{GevreyI}
\bysame, \emph{S\'{e}ries {G}evrey de type arithm\'{e}tique. {I}.
  {T}h\'{e}or{\`e}mes de puret\'{e} et de dualit\'{e}}, Ann. of Math. (2)
  \textbf{151} (2000), no.~2, 705--740.

\bibitem{GevreyII}
\bysame, \emph{S\'{e}ries {G}evrey de type arithm\'{e}tique. {II}.
  {T}ranscendance sans transcendance}, Ann. Math. (2) \textbf{151} (2000),
  no.~2, 741--756.

\bibitem{beukers}
F.~Beukers, \emph{{A refined version of the Siegel--Shidlovskii theorem}}, Ann.
  of Math. (2) \textbf{163} (2006), no.~1, 369--379.

\bibitem{BH}
F.~Beukers and G.~Heckman, \emph{Monodromy for the hypergeometric function
  {$_nF_{n-1}$}}, Invent. math. \textbf{95} (1989), no.~2, 325--354.

\bibitem{BlochEsnault}
S.~Bloch and H.~Esnault, \emph{Gau\ss-{M}anin determinant connections and
  periods for irregular connections}, Geom. Funct. Anal. (2000), no.~Special
  Volume, Part I, 1--31, GAFA 2000 (Tel Aviv, 1999).

\bibitem{BlochEsnault2}
\bysame, \emph{Homology for irregular connections}, J. Th\'{e}or. Nombres
  Bordeaux \textbf{16} (2004), no.~2, 357--371.

\bibitem{chudnovsky}
D.~V. Chudnovsky and G.~V. Chudnovsky, \emph{Applications of {P}ad\'{e}
  approximations to {D}iophantine inequalities in values of {$G$}-functions},
  Number theory ({N}ew {Y}ork, 1983--84), Lecture Notes in Math., vol. 1135,
  Springer-Verlag, Berlin, 1985, pp.~9--51.

\bibitem{DMR}
P.~Deligne, B.~Malgrange, and J-P. Ramis, \emph{Singularit\'{e}s
  irr\'{e}guli{\`e}res}, Documents Math\'{e}matiques (Paris), vol.~5,
  Soci\'{e}t\'{e} Math\'{e}matique de France, Paris, 2007, Correspondance et
  documents.

\bibitem{Dwork}
B.~Dwork, G.~Gerotto, and F.~J. Sullivan, \emph{An introduction to
  {$G$}-functions}, Ann. of Math. Stud., vol. 133, Princeton University Press,
  Princeton, NJ, 1994.

\bibitem{Eisenstein}
G.~Eisenstein, \emph{{{\"U}ber eine allgemeine Eigenschaft der
  Reihen-Entwicklungen aller algebraischen Funktionen}}, Bericht K{\"o}nigl.
  Preuss. Akad. Wiss. Berlin (1852), 441--443.

\bibitem{RivoalFischlerMicro}
S.~Fischler and T.~Rivoal, \emph{Microsolutions of differential operators and
  values of arithmetic {G}evrey series}, Amer. J. Math. \textbf{140} (2018),
  no.~2, 317--348.

\bibitem{RivoalFischler}
\bysame, \emph{On {S}iegel's problem for {$E$}-functions}, Rend. Semin. Mat.
  Univ. Padova, to appear, \url{https://arxiv.org/pdf/1910.06817.pdf}.

\bibitem{FJ}
J.~Fres{\'a}n and P.~Jossen, \emph{Exponential motives}, preliminary version
  available at \url{http://javier.fresan.perso.math.cnrs.fr/expmot.pdf}.

\bibitem{galoshkin}
A.~I. Galochkin, \emph{Criterion for membership of hypergeometric {S}iegel
  functions in a class of {$E$}-functions}, Mat. Zametki \textbf{29} (1981),
  no.~1, 3--14, 154.

\bibitem{gorelov1}
V.~A. Gorelov, \emph{On the {S}iegel conjecture for the case of second-order
  linear homogeneous differential equations}, Mat. Zametki \textbf{75} (2004),
  no.~4, 549--565.

\bibitem{gorelov2}
\bysame, \emph{On the structure of the set of {$E$}-functions satisfying
  second-order linear differential equations}, Mat. Zametki \textbf{78} (2005),
  no.~3, 331--348.

\bibitem{Ince}
E.~L. Ince, \emph{Ordinary {D}ifferential {E}quations}, Dover Publications, New
  York, 1944.

\bibitem{KatzGalois}
N.~M. Katz, \emph{On the calculation of some differential {G}alois groups},
  Invent. math. \textbf{87} (1987), no.~1, 13--61.

\bibitem{Katz}
\bysame, \emph{Exponential sums and {D}ifferential {E}quations}, Ann. of Math.
  Stud., vol. 124, Princeton University Press, Princeton, NJ, 1990.

\bibitem{Lang}
S.~Lang, \emph{Algebra}, third ed., Grad. Texts in Math., vol. 211,
  Springer-Verlag, New York, 2002.

\bibitem{SingerVanDerPut}
M.~van~der Put and M.~F. Singer, \emph{{Galois Theory of Linear Differential
  Equations}}, Grundlehren Math. Wissen., vol. 328, Springer-Verlag, Berlin,
  2003.

\bibitem{Rivoal}
T.~Rivoal, \emph{Les {$E$}-fonctions et {$G$}-fonctions de {S}iegel},
  P{\'e}riodes et transcendance, Ed. {\'E}c. Polytech., Palaiseau, to appear,
  \url{https://rivoal.perso.math.cnrs.fr/articles/EGxups.pdf}.

\bibitem{RivRoq}
T.~Rivoal and J.~Roques, \emph{{Siegel's problem for $E$-functions of order
  $2$}}, {Transcendence in Algebra, Combinatorics, Geometry and Number Theory}
  (A.~Bostan and eds. K.~Raschel, eds.), vol. 373, Springer Proc. Math. Stat.,
  to appear, pp.~473--488.

\bibitem{SiegelShidlovsky}
A.~B. Shidlovskii, \emph{A criterion for algebraic independence of the values
  of a class of entire functions}, Izv. Akad. Nauk SSSR. Ser. Mat. \textbf{23}
  (1959), 35--66.

\bibitem{Shidlovsky}
\bysame, \emph{Transcendental numbers}, De Gruyter Stud. Math., vol.~12, Walter
  de Gruyter \& Co., Berlin, 1989, translated from the Russian by Neal Koblitz;
  with a foreword by W. Dale Brownawell.

\bibitem{Siegel1929}
C.~L. Siegel, \emph{{{\"U}ber einige Anwendungen diophantischer
  Approximationen}}, {A}bhandlungen der {P}reu\ss ischen {A}kademie der
  {W}issenschaften. {P}hysikalisch-mathematische {K}lasse \textbf{1} (1929),
  reprinted in Gesammelte Abhandlungen I, 209--266.

\bibitem{Siegel}
\bysame, \emph{Transcendental numbers}, Ann. of Math. Stud., vol.~16, Princeton
  University Press, Princeton, N. J., 1949.

\bibitem{Vovkodav}
I.~F. Vovkodav, \emph{Logarithmic solutions of higher-order hypergeometric
  differential equations}, Ukrainian Mathematical Journal (1967), 478--482.

\end{thebibliography}

\end{document}